\documentclass[12pt,reqno]{amsart}
       
\usepackage[left=1in, right=1in, top=1.1in,bottom=1.1in]{geometry}
\usepackage[dvipsnames]{xcolor}
\usepackage[mathscr]{eucal}
\usepackage{latexsym,amsfonts,amssymb}
\usepackage{bm,pifont,stmaryrd}
\usepackage{graphicx}
\usepackage{color}
\usepackage{enumitem}
\usepackage{hyperref}
\hypersetup{
     colorlinks   = true,
     citecolor    = blue,
     linkcolor    = blue
} 

\usepackage{pdfsync}
\usepackage[font={scriptsize}]{caption}

\usepackage{tikz}
\usetikzlibrary{trees}

\definecolor{darkblue}{rgb}{0.1, 0.2, 0.75}
\definecolor{darkgreen}{rgb}{0.1, 0.35, 0}
\definecolor{candypink}{rgb}{0.89, 0.44, 0.48}
\definecolor{deepcerise}{rgb}{0.85, 0.2, 0.53}
\definecolor{atomictangerine}{rgb}{1.0, 0.6, 0.4}
\definecolor{fandango}{rgb}{0.71, 0.2, 0.54}

\newcommand{\ep}{\varepsilon}

\newcommand{\DD}{\mathbb{D}}

\newtheorem{theorem}{Theorem}[section]
\newtheorem{Def}[theorem]{Definition}
\newtheorem{notation}[theorem]{Notation}
\newtheorem{thm}[theorem]{Theorem}
\newtheorem{proposition}[theorem]{Proposition}
\newtheorem{prop}[theorem]{Proposition}
\newtheorem{cor}[theorem]{Corollary}
\newtheorem{lemma}[theorem]{Lemma}

\theoremstyle{remark}

\numberwithin{equation}{section}

\newcommand{\di}{\diamond}

\newcommand{\id}{\text{Id}}

\newcommand{\1}{{\bf 1}}

\def\R{\mathbb{R}}
\def\RR{\mathbb{R}}

\def\NN{\mathbb{N}}
\def\mE{\mathbb{E}}
\def\EE{\mathbb{E}}
\def\PP{\mathbb{P}}

\def\bfw{{\bf w}}
\def\bfx{{\bf x}} 
\def\bfy{{\bf y}}

\newcommand{\cb}{{\mathcal B}}

\newcommand{\ce}{{\mathcal E}}
\newcommand{\cf}{{\mathcal F}}
\newcommand{\cg}{{\mathcal G}}
\newcommand{\ch}{{\mathcal H}}
\newcommand{\ci}{{\mathcal I}}
\newcommand{\cj}{{\mathcal J}}
\newcommand{\ck}{{\mathcal K}}
\newcommand{\cl}{{\mathcal L}}
\newcommand{\cm}{{\mathcal M}}
\newcommand{\cn}{{\mathcal N}}

\newcommand{\cp}{{\mathcal P}}

\newcommand{\crr}{{\mathcal R}}
\newcommand{\cs}{{\mathcal S}}

\newcommand{\cv}{{\mathcal V}}

\newcommand{\cx}{{\mathcal X}}

\def\la{{\lambda}}
\def\si{\sigma}

\def\la{{\lambda}}
\def\La{{\Lambda}}

\def\om{{\omega}}
\def\al{{\alpha}}

\def\be{{\beta}}
\def\Ga{{\Gamma}}
\def\ga{{\gamma}}

 \newcommand{\vp}{\varphi}

\newcommand{\lcl}{\left\{}
\newcommand{\rcl}{\right\}}
\newcommand{\lp}{\left(}
\newcommand{\rp}{\right)}
\newcommand{\lc}{\left[}
\newcommand{\rc}{\right]}

\def \eref#1{\hbox{(\ref{#1})}}

\def\tvr{\text{-var}}

\def\ll{\llbracket}
\def\rr{\rrbracket}

\def \eref#1{\hbox{(\ref{#1})}}

\begin{document}
\title[Euler scheme for fBm driven SDEs]
{Euler scheme for SDEs driven\\
 by   fractional Brownian motions: \\integrability and convergence in law}
\date{}    

\author[J. A. Le\'on]{Jorge A. Le\'on}
\address{J. A. Le\'on: Departamento de Control Autom\'atico, Cinvestav-IPN, Mexico}
\email{jleon@ctrl.cinvestav.mx}

\author[Y. Liu]{Yanghui Liu$^{\ast}$}
\address{Y. Liu: Baruch College, New York}
\email{yanghui.liu@baruch.cuny.edu}

\author[S. Tindel]{Samy Tindel}
\address{S. Tindel: Department of Mathematics, 
Purdue University, West Lafayette}
\email{stindel@purdue.edu}

\begin{abstract}We prove that the modified Euler scheme for stochastic differential equations driven by fractional Brownian motions (fBm) with Hurst parameter $H>1/3$, together with its Malliavin derivatives, are integrable  uniformly with respect to the  step size $n$.  Then we use the integrability results  to derive    the   convergence rate in law $n^{1-4H+\ep} $ for the Euler scheme. The proof for integrability is based on a nontrivial generalization of  the greedy sequence argument in \cite{CLL} to a quadratic functional of the fBm.   
The proof of weak convergence applies      Malliavin calculus and  some upper-bound estimates for weighted random sums. 
\end{abstract}

\keywords{Rough paths, Discrete sewing lemma,  Fractional Brownian motion,  Stochastic differential equations,  Euler scheme,      Asymptotic error distributions, Malliavin calculus. 
  \\
 $\ast$ Corresponding author. }

\maketitle

{
}

\section{Introduction}

 This note is concerned with the following stochastic  differential equation driven by a $d$-dimensional fractional Brownian motion (fBm in the sequel) $x$ with Hurst parameter $ \frac13 <   H < \frac12$:
 \begin{eqnarray}\label{e1.1}
 dy_{t}=V_{0}(y_{t})dt+ V(y_{t})dx_{t}\,,
\qquad
y_{0} = a, 
\end{eqnarray}
where   we assume that $a\in \RR^{m} $,  the collection of vector field $V_{0} = (V_{0}^{i})_{1\leq i \leq m}$ belongs to $C^{2}_{b} (\RR^{m}, \RR^{m})  $ and ${V} = (V^{i}_{j})_{1\leq i\leq m, 1\leq j\leq d}$ sits in $C^{3}_{b} (\RR^{m}, \cl (\RR^{d}, \RR^{m})) $. Under this setting   the theory of rough paths   gives a framework allowing to get existence and uniqueness results for  equation~\eref{e1.1}, and the unique solution $y$ in the rough paths sense has $\ga$-H\"older continuity for all $\ga<H$; see e.g. \cite{FH, FV}.

One of the basic questions about systems like \eref{e1.1} concerns the existence of a proper numerical scheme approximating the solution $y$. In case of a Hurst parameter $ H\in(\frac13, \frac12)$, the simplest possible solution to this problem is to use a Milstein type scheme. However Milstein type schemes involve second order expansions and iterated integrals of the fBm $x$, which should be morally thought of as objects of the form 
\begin{equation*}
x^{2}_{st} = \int_{s}^{t}\int_{s}^{u} dx_{v} \otimes dx_{u},
\end{equation*}
and are notoriously uneasy to simulate. Therefore several contributions aimed in the recent past at avoiding iterated integrals while still producing convergent numerical schemes for rough differential equations. The first article tackling this issue is \cite{DNT}, where the iterated integrals in $x^{2}$ were replaced by products of increments of $x$. The rate of convergence obtained in \cite{DNT} was then pushed to its optimal limit in \cite{FR}. Let us also mention the article~\cite{RR}, which thoroughly explores Runge-Kutta methods based on the same idea of replacing iterated integrals by products of increments.

In this paper we will focus our attention on another numerical approximation, called first-order scheme in the sequel. The main idea behind this method is to simply replace the second order terms $x^{2}$ by their expected values. This yields simpler schemes than the aforementioned methods based on product of increments, and at the same time produces optimal convergence rates. Specifically, if $x$ is a fBm with Hurst parameter $H$ and one uses an approximating grid with mesh of order $1/n$, then the rate of convergence is of order $1/n^{2H-1/2}$. This method has first been introduced in~\cite{HLN1} for a Hurst parameter $H>1/2$, and has been extended to the rough path case in~\cite{LT}. We also refer to~\cite{HLN,HLN2} for further extensions.

In order to   describe our first-order numerical scheme, let us   introduce   some basic settings. 
For simplicity, we are considering a finite time interval $[0,T]$ and we take  the uniform partition $\pi: 0=t_{0}<t_{1}<\cdots<t_{n}=T$ on $[0,T]$. Specifically, for $k=0,\ldots,n$ we have $t_{k} =  k \Delta$, where we denote $\Delta = \frac{T}{n}$. In the sequel, the quantity $\delta x_{st}$ will stand for the vector $x_{t}-x_{s}$. Our generic approximation is called $y^{n}$, and it starts from the initial condition $y^{n}_{0} = y_{0}=a$. 
%
With this notation in hand, we can now define our scheme recursively as follows (here and below we set $\delta x_{t_{k}t_{k+1}} = x_{t_{k+1}} - x_{t_{k}}$):
\begin{equation}\label{e4}
y^{n}_{t_{k+1}} = y^{n}_{t_{k}} + V_{0}(y^{n}_{t_{k}}) \Delta + V(y^{n}_{t_{k}}) \delta x_{t_{k}t_{k+1}} 
+\frac12 \sum_{ j=1}^{d}\partial V_{j} V_{j} (y^{n}_{t_{k}}) \, \Delta^{2H}, 
\end{equation}
where  the notation  $\partial V_{i} V_{j}$  stands for a vector field of the form
\begin{eqnarray}\label{not:iterated-vector-field} 
\partial V_{i} V_{j} =
\lp
 \sum_{l=1}^{m} \partial_{l} V_{i}^{k} V_{j}^{l}; \, k=1,\dots, m \, ,
 \rp 
\end{eqnarray}
 and $\partial_{l}$ stands for the partial derivative in the $y_{l}$ direction: $\partial_{l} = \frac{\partial }{\partial {y_{i}} }$. 
As mentioned above, the rate of convergence of $y^{n}$ to $y$ is of order $1/n^{2H-1/2}$. One of the key results in \cite{LT} is a functional central limit theorem of the form
\begin{equation*}
\lim_{n\to\infty} n^{2H-1/2} \lp y^{n} - y \rp
\stackrel{(d)}{=} U,
\end{equation*}
where $U$ is solution to a rough differential equation driven by $x$ plus an additional Brownian term.

In the current contribution, we are mostly interested in the convergence in distribution  of the approximation $y^{n}$ defined by \eqref{e4}. This endeavor is motivated by three main reasons which can be summarized as follows:

\begin{enumerate}[wide, labelwidth=!, labelindent=0pt, label=(\roman*)]
\setlength\itemsep{.1in}

\item The weak convergence of a numerical scheme is   directly related to the performance of simulation for stochastic models, which  is a center issue in mathematical finance and engineering. 

\item 
For diffusions processes, that is stochastic differential equations driven by a Brownian motion, the convergence in distribution  for numerical schemes is a classical problem. This is assessed e.g by the remarkable publications \cite{BT1,BT2}. As mentioned in those two references, a good knowledge about the weak convergence is useful in order to evaluate probabilities that $y$ reaches a certain level, or to get some information about the moments of $y$.

\item
In~\cite{bau,CHLT, geng} we have started a long term program aiming at understanding the law of Gaussian rough differential systems. The current result might play an important role in this approach. 
\end{enumerate}


Let us now describe the main result contained in this paper.

\begin{thm}\label{thm:weak-convergence}
Suppose that  $V\in C^{4}_{b}$ and   $x$ is a fBm with Hurst parameter $H>1/3$. Let $y$ be the solution of the rough differential equation \eqref{e1.1} and let  $y^{n}$
be the corresponding   Euler scheme   
 \eqref{e4}.    
 Then 
for any $\ep>0$,  $f\in C^{4}_{b}(\RR^{m})$ and $t\in [0,T]$ there is a constant $C>0$ independent of $n$ such that
\begin{equation}\label{eq:weak-rate}
\left|\mE
f(y^{n}_{t}) -\mE f(y_{t})
\right|
\leq \frac{C}{n^{4H-1-\ep}}. 
\end{equation}
\end{thm}

\noindent
To the best of our knowledge, Theorem \ref{thm:weak-convergence} is the first weak convergence result for numerical schemes of differential equations driven by a fBm with $H<1/2$. In order to get a broader perspective on weak convergence for stochastic differential systems, let us recall some of the rates obtained in previous contributions:
\begin{enumerate}[wide, labelwidth=!, labelindent=0pt, label=(\alph*)]
\setlength\itemsep{.1in}

\item
It is well-known that the weak convergence rate for an equation like \eqref{e1.1} driven by a Brownian motion  is $n^{-1}$, versus a rate $n^{-1/2}$ for the strong rate; This has been established in the classical references \cite{BT1,BT2}.

\item
 In \cite{HLN} the authors consider differential equations driven by a fBm for the range of Hurst parameter  $H\in(1/2,1)$. It is shown that the weak rate  is   $n^{-1}$ like in the Brownian case, regardless of the value of $H$. This rate is  optimal in the sense that the normalized error $n[\mE
f(y^{n}_{t}) -\mE f(y_{t})]$ converges to a nonzero limit for any test function $f\in C^{4}$.

\item
The recent articles~\cite{FSW, Ga} consider the weak convergence of Euler schemes for a mixed stochastic integral model  
$\ci= \int_{0}^{T} \vp\lp  B_{t}^{H} \rp \, dW_{s}$, where $W$ is a Wiener process and $B^{H}$ is  a Liouville type fractional Brownian motion driven by $W$ (with Hurst parameter  $H\in(0,1)$). It is proved that for a general choice of functions $f,\vp$, the weak rate is $1/n^{(3H+1/2)\wedge 1}$ (to be contrasted with the strong rate of the Euler scheme for $\ci$ above, which is $1/n^{H}$).  Our result   shows that this surprising behavior is probably due to some specific cancellations for mixed quantities like~$\ci$ (see further remarks about this fact in~\cite{FSW}).

\end{enumerate}

\noindent
Compared to this body of literature, our Theorem \ref{thm:weak-convergence}   shows that when $1/3<H<1/2$ the weak rate for equation~\eqref{e1.1}  is $n^{1-4H} $ (note that  we believe that our rate is  optimal  for a generic  test function).  
This generalizes in a very natural way the convergence rate $n^{-1}$ obtained for $H\ge 1/2$, except for the slightly non optimal $\ep$ in relation~\eqref{eq:weak-rate}. Notice that this small $\ep$ is due to the fact that our analysis of the scheme is mostly pathwise, in spite of dealing with a convergence in distribution.
 It is interesting to mention that 
   Theorem \ref{eq:weak-rate}  agrees with the rule of thumb  in the martingale framework (see e.g. Heston's model~\cite{AN}, Schr\"odinger's equation~\cite{DD}, reflected diffusions~\cite{BGT} or the stochastic heat equation~\cite{DP}), namely  that  the weak rate  $n^{1-4H} $ is twice the strong rate $n^{1/2-2H}$ (see \cite{LT}).



At the core of our methodology for the proof of Theorem~\ref{thm:weak-convergence} lies a combination of rough paths and Malliavin techniques, plus some specific tools for discrete rough paths that have been developed by two of the authors in~\cite{LT,LT2}. Those elements are summarized in Section~\ref{section.2} and Sections \ref{section4.1}-\ref{section4.2}. 
Our main additional ingredient in the current contribution is to show the integrability of the Malliavin derivatives of the Euler scheme \eqref{e4}, uniformly in our approximation parameter $n$. A key observation in  this direction is that the Euler scheme ~\eqref{e4} is a discrete-time equation driven by the mix of a rough path  (that is the process $x$) and a quadratic \textit{Young} path (that is, a path  which is a quadratic functional of $x$ and  has a  H\"older component greater than $1/2$; see \eqref{eqn.q} for the precise definition). This  representation enables us to adapt the very fruitful idea of \emph{greedy sequence} put forward in~\cite{CLL}, in order to achieve exponential integrability in a rough paths context.  
A new situation for the Euler scheme is that now we have a greedy sequence corresponding  not only to $x$ but also to the quadratic path $q$ introduced in~\eqref{eqn.q}. One of our main efforts will then consist in  showing a tail estimate for the greedy sequence via  Borell's inequality.   
    Furthermore, due to the discrete feature of   equation \eqref{e4}, a separate estimate will involve the \emph{big} steps related to our partition of $[0,T]$ (namely the steps for which the increments $\delta x_{t_{k}t_{k+1}}$ are very large) separately. These delicate estimates will be developed in Section~\ref{section.3}.
    
The paper is structured as follows. Section~\ref{section.2} contains the  preliminary results on rough paths, Malliavin calculus, and the Euler scheme. In Section~\ref{section.3} we show that   the Malliavin derivatives of the Euler scheme has moments of any order.
After some   preparations in Section \ref{section4.1}-\ref{section.decompose},   we prove the weak convergence of the Euler scheme in Section \ref{section.weak}.

\begin{notation}\label{general-notation}
In what follows,  we take $n\in \NN$ and   $\Delta =T/n$,  and  consider the uniform partition:   
  $0=t_{0}<t_{1}<\cdots<t_{n}=T$ on $[0, T]$, where  $t_{k} =k\Delta$. We denote by   $\ll s, t\rr$ the discrete interval: $\ll s, t\rr = \{t_{k}\in [s,t]: k=0,\dots, n\}$. For $u\in [t_{k}, t_{k+1})$, we denote $\eta(u) =t_{k}$. 
 For an interval $[s,t]\subset [0,T]$  we define the continuous-  and discrete-time simplexes  $\cs_{2}([s,t]) = \{(u,v): s\leq u\leq v \leq t\}$ and $\cs_{2}(\ll s,t\rr) = \cs_{2}([s,t]) \cap \ll s,t\rr^{2}$.  
 We use the letters $C$ and $K$ to denote generic constant which can change from line to line.

\end{notation}
 
\section{Preliminary results} \label{section.2}

In this section we recall some basic notions of rough paths theory and their application to fractional Brownian motion, which allow a proper definition of equation~\eqref{e1.1}. We also give the necessary elements of Malliavin calculus in order to 
quantify the weak convergence rate. Eventually we recall the pathwise estimates obtained in \cite{LLT2} for the Malliavin derivatives of our Euler scheme. Notice that this basic presentation can be found in a very similar way in our companion paper \cite{LLT2}.
  
\subsection{Elements of rough paths theory}\label{section2}

This subsection is devoted to introduce some basic concepts of rough paths theory. We are going to restrict our analysis to a generic $p$-variation  regularity of the driving path of order $1\leq p<3$, in order to keep expansions to a reasonable size. We also fix a finite time horizon $T>0$. The following notation will prevail until the end of the paper: for a finite dimensional vector space $\cv$ and two functions $f\in C([0,T],\cv)$ and $g\in C(\cs_{2}([0,T]),\cv)$ we set 
\begin{eqnarray}\label{eq:def-delta}
\delta f_{st} = f_{t}-f_{s},
\quad\text{and}\quad
\delta g_{sut} = g_{st}-g_{su}-g_{ut}.
\end{eqnarray}

Let us introduce the analytic requirements in terms of $p$-variation regularity which will be used in the sequel.  Namely consider two paths $x \in C([0,T], \RR^{d})$ and $x^{2} \in C(\cs_{2}([0,T]), (\RR^{d})^{\otimes 2})$.  Then we denote
\begin{equation}\label{eq:def-holder-seminorms}
\|x\|_{  p\text{-var}, [s,t]}:=\lp
\sup_{\cp}\sum_{(u,v)\in\cp} {|\delta x_{uv}|^{p}}\rp^{1/p} , 
\quad  
\|x^{2}\|_{  p/2\text{-var}, [s,t]}:= \lp\sup_{\cp}\sum_{(u,v)\in\cp} {|  x^{2}_{uv}|^{p/2} }\rp^{2/p} ,
\end{equation}
where the supremum is taken among all partitions   of the time interval $[s,t]$, and for a partition $\cp$ of $[s,t]$ we write $(u,v)\in\cp$ if $u$ and $v$ are two consecutive partition points of $\cp$. When the semi-norms in \eqref{eq:def-holder-seminorms} are finite we say that $x$ and $x^{2}$ are respectively in $C^{p\text{-var}}([s,t], \RR^{d})$ and $C^{p/2\text{-var}}(\cs_{2}([s,t]), (\RR^{d})^{\otimes 2})$.
For convenience, we denote $ \|x\|_{p\text{-var}}:= \|x\|_{ p\text{-var}, [0,T]} $ and $ \|x^{2} \|_{p/2\text{-var} }:= \|x^{2}\|_{ p/2\text{-var}, [0,T] } $.
With this preliminary notation in hand, we can now turn to the definition of rough path.

\begin{Def}\label{def:rough-path}
Let $x \in C([0,T], \RR^{d})$, $x^{2} \in C(\cs_{2}([0,T]), (\RR^{d})^{\otimes 2})$, and $1\leq  p< 3$. Denote $x^{1}_{st}=\delta x_{st}$.  
We call $ \bfx:=S_{2}(x):=(x^{1}, x^{2})  $ a (second-order) $p$-rough path if $ \|x\|_{p\text{-var}} <\infty $ and $\|x^{2}\|_{p/2\text{-var}}<\infty$, and if the following algebraic relation holds true: 
\begin{eqnarray}\label{eqn.delta.def}
\delta x^{2}_{sut} =x^{2}_{st} - x^{2}_{su} - x^{2}_{ut} = \delta x_{su}\otimes \delta x_{ut} ,
\end{eqnarray}
where we have invoked \eqref{eq:def-delta} for the definition of $\delta x$ and $\delta x^{2}$. For a $p$-rough path $S_{2}(x)$, we define a $p$-variation  semi-norm as follows:
\begin{eqnarray}\label{eq:def-norm-rp}
\|S_{2}(x)\|_{p\text{-var}} := \|x\|_{p\text{-var}}+ \|x^{2}\|_{p/2\text{-var}}^{1/2} \,. 
\end{eqnarray}
An important subclass of rough paths are the so-called \emph{geometric $p$-variation  rough paths}.  A geometric $p$-variation  rough path is a  $p$-rough path $  (x, x^{2})$  such that  there exists a sequence of smooth $\RR^{d}$-valued paths $(x^{n}, x^{2,n})$ verifying:
\begin{eqnarray}\label{eq:cvgce-for-geom-rp}
\lim_{n\to\infty} \lp 
\| x-x^{n}\|_{p\text{-var}} + \|x^{2}-x^{2,n}\|_{p/2\text{-var}} 
\rp
 = 0. 
\end{eqnarray}
We will mainly consider geometric rough paths in the remainder of the article. 
\end{Def}

\noindent
In relation to \eqref{eq:cvgce-for-geom-rp}, notice that  when $x$ is a smooth  $\RR^{d}$-valued path, we can choose $x^{2}$ defined as the following iterated Riemann type integral,
\begin{eqnarray}\label{e2.2}
x^{2}_{st} &=& \int_{s}^{t}\int_{s}^{u} dx_{v} \otimes dx_{u}.
\end{eqnarray}
It is then easily verified  that $S_{2}(x) = (x^{1}, x^{2})$, with $x^{2}$ defined in \eref{e2.2}, is a $p$-rough path with $p=1$.  In fact, this is also the unique way to lift a smooth path to a $p$-rough path for some $p\geq 1$.

Recall now that we interpret equation \eqref{e1.1} in the rough paths sense. That is, we shall consider the following general rough differential equation (RDE):
\begin{eqnarray}
\label{e2.1}
y_{t}&=&a+\int_{0}^{t}V_{0}(y_{s})ds + \int_{0}^{t} {V}(y_{s}) d x_{s}\,, \quad t\in [0,T],
\end{eqnarray}
where $V_{0}$ and $V$ are smooth enough coefficients and $x$ is a rough path as given in Definition~\ref{def:rough-path}. We shall interpret equation \eqref{e2.1} in a way introduced by Davie in \cite{D}, which is conveniently compatible with numerical approximations.

 \begin{Def}\label{def:diff-eq-davie}
Let $(x, x^{2})$ be a  $p$-rough path with $p<3$. We say that $y$ is a solution of~\eref{e2.1} on $[0,T]$ if $y_{0} = a$ and there exists a control function  $\omega$ on $[0,T]$ (that is, $\omega$ is a two variable function  on $\cs_{2}([ 0, T])$ which satisfies the super-additivity condition $\omega(s,t)\geq \omega(s,u)+\omega(u,t)$ for $s,u,t\in[0,T]: s<u<t$), a constant $K>0$ and $\mu>1$  such that 
\begin{equation}\label{eq:dcp-Davie}
\Big| \delta y_{st} -  \int_{s}^{t} V_{0}(y_{u}) \, du - V(y_{s}) \delta x_{st} 
- \sum_{i,j=1}^{d} \partial V_{i}V_{j} (y_{s} ) x^{2,ij}_{st} \Big| 
\leq 
K \omega(s,t)^{\mu}
\end{equation}
for all $(s,t) \in \mathcal{S}_{2}([0,T])$, where we recall that $\delta y$ is defined by~\eref{eq:def-delta} and $\partial V_{i}V_{j}$ is defined     as in \eqref{not:iterated-vector-field}.  
\end{Def}

\noindent
Notice that if $y$ solves \eqref{e2.1} according to Definition \ref{def:diff-eq-davie}, then it is also a controlled process as defined in \cite{FH,G}. Namely, if $y$ satisfies  relation \eqref{eq:dcp-Davie}, then we also have:
\begin{eqnarray*}
\delta y_{st} = V(y_{s}) \delta x_{st} + r_{st}^{y},
\end{eqnarray*}
where $r^{y}\in C^{p/2\text{-var}}(\cs_{2}([0,T]))$. We can thus define iterated integrals of $y$ with respect to itself thanks to the sewing map; see   Proposition 1 in \cite{G}. This yields the following decomposition:
\begin{eqnarray*}
\Big| 
\int_{s}^{t}   y_{u}^{i} d y^{j}_{u} - y^{i}_{s} \delta y^{j}_{st} -  \sum_{i', j'=1}^{d}  V^{i}_{i'}    V^{j}_{j'}(y_{s}) x_{st}^{2,i'j'}
\Big|&\leq& K\omega(s,t)^{3/p},
\end{eqnarray*}
for all $(s,t) \in \mathcal{S}_{2}([0,T])$ and $i,j=1,\ldots,m$. In other words, the signature type path $S_{2}(\bfy) = (y^{1}, y^{2} )$ defines a rough path according to Definition \ref{def:rough-path},  where $y^{2}$ denotes the iterated integral of $y$.

We can now state an existence and uniqueness result for rough differential equations. The reader is referred  to e.g. \cite[Theorem 10.36]{FV}   for further details.
\begin{thm}\label{thm 3.3}
Assume that  $V= (V_j)_{1\leq j\leq d}$ is a collection of  $C^{3}_{b}$-vector fields on $\RR^m$.
 Then there exists a unique   RDE solution to equation \eref{e2.1}, understood as in Definition~\ref{def:diff-eq-davie}. In addition, there exists a constant $K>0$ such that the unique solution $y$ satisfies the following estimate:
 \begin{eqnarray*}
|  S_{2}(y)_{st}| &\leq & K \lp 1\vee \| S_{2}(x) \|_{p\text{-var}, [s,t] }^{p} \rp  .
\end{eqnarray*}
Whenever   $V = (V_j)_{1\leq j\leq d} $ is a collection of linear vector fields,
existence and uniqueness still hold for equation \eqref{e2.1}. Furthermore, there exist constants  $K_{1},K_{2}>0$ such that    we have the estimate:
    \begin{eqnarray*}
|S_{2}(y)_{st}| &\leq& K_{1} 
\| S_{2}(x) \|_{p\text{-var}, [s,t] } 
\, \exp\lp K_{2} 
\| S_{2}(x) \|_{p\text{-var} }^{p}
  \rp  \, .
\end{eqnarray*}

 \end{thm}

We close this section by recalling a sewing map lemma with respect to discrete control functions. It is an elaboration  of \cite[Lemma 2.5]{LT} and proves to be useful in the analysis of the numerical scheme.  Let  $\pi : 0=t_{0}< t_{1}<\cdots <t_{n-1}< t_{n} =T $ be a generic partition of the interval $[0,T]$ for $n \in \NN$. We denote by $ \llbracket  s, t \rrbracket  $ the discrete interval  $\{t_{k} : s\leq t_{k} \leq t \}$ for $0\leq s < t \leq T$.

 \begin{lemma}\label{lem2.4}
 Suppose that $\omega $ is a control on $\ll 0, T\rr$. In other words, $\omega$ is a two variable function  on $\cs_{2}(\ll 0, T\rr)$ which satisfies a super-additivity condition: $\omega(s,t)\geq \omega(s,u)+\omega(u,t)$ for $s,u,t\in\ll0,T\rr: s<u<t$.  
 Consider a Banach space $\cb$ with norm $|\cdot |$ and $R : \cs_{2}(\ll 0, T \rr)   \to \cb $, and 
   denote $\delta R_{sut} = R_{st}-R_{su}-R_{ut}$. 
  Suppose that $|R_{t_{k}t_{k+1}}|\leq\omega(t_{k}, t_{k+1})^{\mu}$ for all $t_{k}\in \ll0,T\rr$, and that $|\delta R_{sut}|\leq \omega (s, t)^{\mu}$ with  the exponent $\mu>1$. Then the following relation holds:
 \begin{eqnarray}\label{eqn.kmu}
 |R_{st} |  \leq K_{\mu} \omega(s, t)^{\mu} \,,
\quad\text{where}\quad
K_{\mu} = 2^{\mu} \, \sum_{l=1}^{\infty} l^{-\mu}.
\end{eqnarray}
\end{lemma}
 
 The discrete sewing lemma allows to bound discrete sums which are crucial in our numerical scheme context. As a first application along those lines  we  present a probabilistic result below, which combines Proposition {4.1} and Remark {4.2} in \cite{LT}. 
 
 \begin{lemma}\label{lem.y}
Consider two processes $f$ and $g$ such that for all $s,t\in \ll 0, T \rr$ we have 
\begin{eqnarray*}
\|\delta f_{st}\|_{L^{2p}} \lesssim |t-s|^{\al}, 
\qquad
\text{and}
\qquad
\|\delta g_{st}\|_{L^{2p}} \lesssim |t-s|^{\be}, 
\end{eqnarray*}
for a given $p\geq 1$ and $\al, \be$ such that $\al+\be>1$. Let $J_{st}$ be the discrete sum given by 
\begin{eqnarray}\label{e.jst}
J_{st} = \sum_{s\leq t_{k}<t} \delta f_{st_{k}} \delta g_{t_{k}t_{k+1}}. 
\end{eqnarray}
Then we have 
\begin{eqnarray*}
\|J_{st}\|_{L^{p}}\lesssim (t-s)^{\al+\be}. 
\end{eqnarray*}

\end{lemma}
 
\subsection{Rough path above fractional Brownian motion}\label{section.rp}

We now specialize our setting to a path $x=(x^1,\dots, x^d)$ defined as a standard $d$-dimensional fBm  on $[0,T]$ with Hurst parameter $H \in (\frac13, \frac12)  $. 
This fBm is defined on a complete probability space $(\Omega, \cf, \mathbb{P})$, and we assume that the $\si$-algebra $\cf$ is generated by $x$. 
 In this situation, recall that the covariance function of each coordinate of $x$ is defined on $\mathcal{S}_{2}([0,T])$ by:
\begin{eqnarray}\label{eq:cov-fbm}
R(s,t) = \frac{1}{2} \lc s^{2H} + t^{2H} - |t-s|^{2H}  \rc ,
\end{eqnarray}
where recall that the simplex $\cs_{2}([0,T])$ is introduced in Notation \ref{general-notation}.
We start by reviewing some properties of the covariance function of $x$ considered as a function on $(\mathcal{S}_{2}([0,T]))^{2}$. Namely, take $(u,v,s,t)$ in $(\mathcal{S}_{2}([0,T]))^{2}$ and set 
\begin{eqnarray}\label{eq3.1}
R ([u,v], [s,t]) 
&=& \mE\lc \delta x^{j}_{uv} \, \delta x^{j}_{st} \rc, 
\qquad j=1,\dots,d.
\end{eqnarray}
Then, whenever $H>1/4$, it can be shown that the integral $\int R \, dR$ is well-defined as a Young integral in the plane  (see e.g. \cite[Section 6.4]{FV}). Furthermore, if the  intervals $[u,v]$ and $[s,t]$ are disjoint, we have 
\begin{eqnarray}\label{eq:covariance-disjoint-intv}
R ([u,v], [s,t]) 
= 
 \int_{u}^{v}\int_{s}^{t} \mu( dr'dr) .
\end{eqnarray}
 Here and in the following, the signed measure $\mu$ is defined as
 \begin{eqnarray}\label{eq:def-msr-mu}
\mu(dr'dr) &=& - H(1-2H) |r-r'|^{2H-2} dr'dr.
\end{eqnarray}

Using the elementary properties above, it is shown in \cite[Chapter 15]{FV} that for any piecewise linear or mollifier approximation $ x^{n} $ to $x$, the smooth rough path $S_{2}(x^{n})$ defined by~\eqref{e2.2} converges  in the $p$-variation semi-norm \eqref{eq:def-norm-rp} to a $p$-geometric rough path $S_{2}(x):=(x^{1}, x^{2})$ {(given as in Definition \ref{def:rough-path})} for $3>p >1/H$.   In addition, for $i\neq j$ the covariance of $x^{2,ij}$ can be expressed in terms of a $2$-dimensional Young integral: 
\begin{eqnarray}\label{e3.1}
\mE\lc x^{2,ij}_{u v} x^{2,ij}_{s t}\rc  
&=&  
\int_{u}^{v} \int_{s}^{t} 
R ([u,r], [s,r'])  
 d  R(r',r)
.
\end{eqnarray}
It is also established in \cite[Chapter 15]{FV} that $S_{2}(x)$ enjoys the following integrability property.
\begin{prop}\label{prop:integrability-signature}
Let $S_{2}(x):=(x^{1}, x^{2})$ be the geometric rough path above $x$ as given in Definition \ref{def:rough-path}, and $p\in(1/H, 3)$. Then there exists a random variable $G_{p}\in \cap_{p\ge 1}L^{p}(\Omega)$ such that $\|S_{2}(x)\|_{p\text{-var}}\le G_{p}$, where $\| \cdot \|_{p\text{-var}}$ is defined by \eref{eq:def-norm-rp}.
\end{prop}

According to Theorem \ref{thm 3.3}, given that  the vector fields $V\in C_{b}^{3}$,   equation (\ref{e2.1}) driven by a $d$-dimensional fBm $x$ with Hurst parameter $H>1/3$ admits a unique solution. 

 
\subsection{Malliavin calculus for $\mathbf{x}$}\label{subsection.d}

As mentioned in the introduction, we will analyze the convergence of distribution  for our numerical approximations thanks to Malliavin calculus tools. We proceed to recall the main concepts which will be used later in the paper and refer to \cite{N} for further details.
We start by labeling a definition for the Cameron-Martin type space $\ch$ related to our fractional Brownian motion $x$.
 
\begin{Def}\label{def3.2}
 Denote by $\mathcal{E}_{[a,b]}$ the set of step functions on an interval $[a,b] \subset [0,T]$. We call $\mathcal{H}_{[a,b]}$ the Hilbert space defined as  the closure of $ \mathcal{E}_{[a,b]} $ 
with respect to the scalar product
\begin{eqnarray*}
\langle \mathbf{1}_{[u,v]}, \mathbf{1}_{[s,t]} \rangle_{\mathcal{H}_{[a,b]}} &=&  R ([u,v], [s,t])  .
\end{eqnarray*}
In order to alleviate notations, we will write $\mathcal{H}=\mathcal{H}_{[a,b]}$ when $[a,b]=[0,T]$. 
Notice that the mapping $\mathbf{1}_{[s,t]} \rightarrow \delta x_{st}$ can be extended to an isometry between $\mathcal{H}_{[a,b]}$ and the Gaussian space associated with $\{ x_{t}, t\in [a,b]\}$. 
 We denote this isometry by $h \rightarrow \int_{a}^{b} h \, \delta^{\di} x $. The random variable $\int_{a}^{b} h \, \delta^{\di} x$   is   called the (first-order) Wiener integral and is also denoted by $I_{1}(h)$. 
 \end{Def}
 
The space $\ch$ is very useful in order to define Wiener integrals with respect to $x$. In this paper we also need to introduce another Cameron-Martin type space $\bar{\ch}$.  The space $\bar{\ch}$ allows to identify pathwise   derivatives with respect to $x$ and the Malliavin derivatives. In order to construct $\bar{\ch}$,  let $\crr$ be the linear operator such that $\crr: h\in\ch\to \langle h, \mathbf{1}_{[0,t]} \rangle_{\ch}$. Then the space $\bar{\ch}$ is defined as the Hilbert space  $\bar{\ch}=\crr(\ch)$ equipped with the inner product 
\begin{eqnarray*}
\langle \crr(g), \crr(h) \rangle_{\bar{\ch}} = \langle g, h \rangle_{\ch}. 
\end{eqnarray*}
We refer to \cite{GOT,NS} for more details about the spaces $\ch$ and $\bar{\ch}$.

For the sake of conciseness, we refer to \cite{N} for a proper definition of Malliavin derivatives and related Sobolev spaces in Gaussian analysis. Let us just mention that we will denote the Malliavin derivative by $ {D}F$, the Sobolev spaces by $\mathbb{D}^{k,p}$ and the corresponding norms by $\|F\|_{k,p}$. We denote by $D^{k}F$ the $k$th iteration of the Malliavin derivative $D$ applied on $F$.  
The $n$-th order chaos of $x$ is denoted by $\ck^{x}_{n}$.  Also notice that we are considering a $d$-dimensional fBm $x=(x^{1}, \dots, x^{d})$. Therefore, we shall consider partial Malliavin derivatives with respect to each coordinate $x^{i} $ in the sequel. 
Those partial derivatives will be denoted by $D^{(i)}$. Then for $h = (h^{1}, \dots, h^{d}) \in \ch^{d}$ we write $D_{h}F = \sum_{i=1}^{d} \langle D^{(i)} F, h^{i} \rangle_{\ch}$.  
For $L\geq 2$ we also denote by $D^{L}_{h}$ the iterated versions of $D_{h}$. Namely we set 
\begin{eqnarray}\label{eqn.dlf}
D^{L}_{h}F = D_{h}\circ\cdots \circ D_{h} F. 
\end{eqnarray}

The Sobolev spaces related to the Malliavin derivatives are denoted by $\DD^{k,p}$ and the corresponding norms are written $\|\cdot\|_{k,p}$. The dual of the Malliavin derivative is the Skorohod integral, for which we use the notation $\delta^{\diamond}$. Its domain includes the space $\DD^{1,2}(\ch^{d})$, and the integration by parts formula can be read as 
\begin{eqnarray}\label{eqn.it}
\mE[ F \delta^{\diamond} (u) ] = \mE [ \langle DF, u \rangle_{\ch^{d}} ], 
\end{eqnarray}
valid for $F\in \DD^{1,2}$ and $u\in \DD^{1,2}(\ch^{d})$.

\subsubsection{Differentiability}\label{section.diff} 
As we will see below, under  the condition that $V\in C_{b}^{\lfloor 1/\ga \rfloor+1}$ the solution $y$ to \eqref{e2.1} is infinitely differentiable in the Malliavin calculus sense. We shall express its Malliavin derivative in terms of the Jacobian $\Phi$ of the equation, which is defined by the relation $\Phi_{t}^{ij}=\partial_{a_j}y_t^{(i)}$, where recall that $a=(a_{1}, \dots, a_{m}) $ is the initial value of the system \eqref{e2.1}. Setting $\partial V_{j}$ for the Jacobian of $V_{j}$ seen as a function from $\R^{m}$ to $\R^{m}$, let us recall that $\Phi$ is the unique solution to the linear equation
\begin{equation}\label{eq:jacobian}
\Phi_{t} = \id_{m} + \int_0^t \partial V_0 (y_s) \, \Phi_{s} \, ds+
\sum_{j=1}^d \int_0^t \partial V_j (y_s) \, \Phi_{s} \, dx^j_s . 
\end{equation}
Moreover, the following results hold true:
\begin{proposition}\label{prop:deriv-sde}
Let $y$ be the solution to equation (\ref{e2.1}) and suppose $(V_{0}, V_{1}, \dots, V_{d})$ is a collection of vector fields in  $ C_{b}^{3}$. Then
for every $i=1,\ldots,m$, $t>0$, and $a \in \mathbb{R}^m$, we have $y_t^{(i)} \in
\mathbb{D}^{2,p}(\ch)$ for $p\geq 1$ and
\begin{equation*}
 {D}^{(j)}_s y_t= \Phi_{s,t} V_j (y_s) , \quad j=1,\ldots,d, \quad 
0\leq s \leq t,
\end{equation*}
where $ {D}^{(j)}_s y^{(i)}_t $ is the $j$-th component of
$ {D}_s y^{(i)}_t$, $\Phi_{t}=\partial_{a} y_t$ solves equation \eqref{eq:jacobian} and $\Phi_{s,t}=\Phi_{t}\Phi_{s}^{-1}$. 
\end{proposition}

Let us now quote the result \cite{CLL}, which gives a useful estimate for moments of the Jacobian of rough differential equations driven by Gaussian processes. Note that this result is expressed in terms of $p$-variations, for which we refer to \cite{FV}.

\begin{proposition}\label{prop:moments-jacobian}
Consider a  fractional Brownian motion $x$ with Hurst parameter $H\in (1/4,1/2]$ and $p>1/H$. Then for any $\eta\ge 1$, there exists a finite constant $c_\eta$ such that the Jacobian $\Phi$ defined by \eqref{eq:jacobian} satisfies:
\begin{equation}\label{eq:moments-J-pvar}
\EE\lc  \Vert \Phi \Vert^{\eta}_{p\text{-}{\rm var}; [0,1]} \rc = c_\eta.
\end{equation}
\end{proposition}

%

%
%
%
%
%
%
%
%
%
%
%
%

\subsection{Path-wise estimate of Euler scheme and its derivatives} 
In the remaining of the section  we state a path-wise upper-bound estimate of the Malliavin derivatives of $y^{n}$ obtained in our companion paper \cite{LLT2}. We first introduce some notations. 

  Let $b$ be a fBm independent of $x$. {Recall that the rough paths above $x$ and $b$ are denoted by $(x^{1}, x^{2})$ and $(b^{1}, b^{2})$, respectively (see Definition \ref{def:rough-path}).} 
   We introduce some second chaos processes which play a prominent role in the analysis of Euler schemes (see \cite{LT}). Namely for $[s,t] \in\ll0,T\rr$ we set 
 \begin{eqnarray}\label{eqn.q}
q_{st}^{ij}= \sum_{s\leq t_{k}<t} \lp   x_{t_{k}t_{k+1}}^{2,ij} -\frac12 \Delta^{2H} \mathbf{1}_{\{i=j\}} \rp \, ,
\quad\text{and}\quad
q^{b, ij}_{st} = \sum_{s\leq t_{k}<t} \lp   b_{t_{k}t_{k+1}}^{2,ij} -\frac12 \Delta^{2H} \mathbf{1}_{\{i=j\}} \rp   .
\end{eqnarray}
We   recall a basic inequality taken from \cite[Lemma 3.4]{LT}: for $(s,t) \in  \cs_{2} ( \ll 0,T \rr )$ we have 
\begin{eqnarray}\label{eqn.f.bd}
\lp \mE [ \|q_{st}\|^{2} ]\rp^{1/2}
\lesssim \frac{(t-s)^{1/2}}{n^{2H-1/2}} \, .
\end{eqnarray}

We also  introduce a Gaussian process $\text{w}$ which encompasses the coordinates of both the driving noise $x$ and the extra noise~$b$. Specifically we define
\begin{eqnarray}\label{eqn.w.def}
\delta \text{w}_{st} := (\delta \text{w}_{st}^{1}, \dots, \delta \text{w}_{st}^{2d}    ):=  (\delta x_{st}^{1}, \dots,\delta x_{st}^{d}, \delta b_{st}^{1},   \dots, \delta b_{st}^{ d}    ) . 
\end{eqnarray}
 Furthermore, we define a control $\omega$ by  
  \begin{align}\label{eqn.control}
  \omega (s,t) = \| \bfw\|_{p\text{-var}; \ll s,t\rr}^{p} + \|q\|_{p/2\text{-var};  \ll s,t\rr }^{p/2} + \|q^{b}\|_{p/2\text{-var};  \ll s,t\rr }^{p/2},  
    \qquad (s,t)\in  \cs_{2}(\ll 0, T\rr),  
\end{align}
where $q$ is defined in \eref{eqn.q}, $\text{w}= (x,b)$ according to \eqref{eqn.w.def} and $\bfw=S_{2}(\text{w})$ is the $p$-rough path above $\text{w}$ (see Definition \ref{def:rough-path}). 
Let $\al >0$ be a positive constant. 
Denote $s_{0}=0$. Then
given $ s_{j}$, we define $s_{j+1}$ recursively as
\begin{equation}\label{eqn.sj}
s_{j+1}=
\begin{cases}
s_{j}+\Delta \, ,
&\textnormal{if } \omega(s_{j}, s_{j}+\Delta) > \al \\
\max  \{u \in \ll 0,T \rr :\, u>s_{j} \text{ and }   \omega(s_{j}, u)\leq \al\} \, ,
&\textnormal{if } \omega(s_{j}, s_{j}+\Delta) \le \al
\end{cases}
\end{equation}
Next we split the set of $s_{j}$'s as   
 \begin{align}
&S_{0} = \{s_{j}:  \al/2\leq \omega(s_{j}, s_{j+1})\leq \al\};
\quad
 S_{1}=  \{s_{j}:    \omega(s_{j}, s_{j+1}  )< \al/2\};
 \label{eqn.s01}
\\
&S_{2}=  \{s_{j}:    \omega(s_{j}, s_{j+1} )> \al\}. 
\label{eqn.cs2}
\end{align}
 We set 
\begin{align}\label{eqn.ms}
&\cm_{0} =  
\prod_{s_{j}\in S_{0} } \lp K \omega(s_{j},s_{j+1})^{1/p}   +1\rp,
\qquad \cm_{1} =  
\prod_{s_{j}\in   S_{1}} \lp K \omega(s_{j},s_{j+1})^{1/p}   +1\rp,
\nonumber
 \\
& \cm_{2} =   
 \prod_{s_{j}\in S_{2} }
  \lp K  |  \delta{\text{w}}_{s_{j}s_{j+1}}| +K\Delta^{2H} +1\rp 
 ,
\end{align}
and $K$ is a constant independent of $n$. 

We now  recall the following path-wise estimates for the Euler scheme in \cite[Theorem 4.13]{LLT2}.

 \begin{theorem}\label{thm.xirr}
 Take $r,r'\geq 0$ and let $k_{0},k_{0}'\in \NN$ be such  that $r\in (t_{k_{0}}, t_{k_{0}+1}] $ and $r'\in (t_{k_{0}'}, t_{k_{0}'+1}] $. Suppose that $V\in C^{4}_{b}$ and $p>1/H$. Define a Malliavin derivative vector $\xi^{n} $ as
\begin{equation}\label{d1}
\xi^{n} _{t}= (y^{n}_{t},D_{r}y^{n}_{t} , D_{r}D_{r'}y^{n}_{t} ) := (\xi_{t}^{n,0},\xi_{t}^{n,1}, \xi_{t}^{n,2})
\end{equation}
 Then for $L=0,1,2 $ and all $(s,t) \in \cs_{2}\ll 0,T \rr$  we have the estimate 
\begin{align}\label{eqn.z.pbd1}
\|\xi^{n,L}\|_{p\tvr, \ll s,t\rr}   \leq  
K \cdot \omega(s,t)^{1/p}\cdot \cg \, ,
\end{align}
where the random variable $\cg$ is defined by
\begin{eqnarray}\label{e.cg}
\cg = |S_{0}\cup S_{1}\cup S_{2}| \cdot (\cm_{0}\cdot \cm_{1}\cdot \cm_{2} )^{L} \, ,
\end{eqnarray}
and where the quantities $S_{i}$,
 $\cm_{i}$ are respectively  defined for $i=0,1,2$ in \eqref{eqn.s01}-\eqref{eqn.cs2} and~\eqref{eqn.ms}. Moreover, for
  $(s,t)\in \cs_{2}(\ll s_{j},s_{j+1} \rr)$ such that  $s_{j}\in S_{0}\cup S_{1}$  we have 
\begin{eqnarray}\label{eqn.dxiestimate2}
|\delta \xi^{n,L}_{st} - \cl^{L} V(y^{n}_{s})    \delta x_{st}
   | \leq    K     \omega(s,t)^{2/p} \cdot \cg^{2},
  \qquad  L=0,1,2, 
\end{eqnarray}
where we have set
\begin{equation*}
\cl^{0}V(y^{n}_{s})=V(y^{n}_{s}), \qquad
\cl^{1}V(y^{n}_{s})=\partial V(y^{n}_{s})D_{r}y^{n}_{s},
\end{equation*}
and where $\cl^{2}V(y^{n}_{s})$ is defined by
\begin{eqnarray*}
\cl^{2}V(y^{n}_{s})=\partial^{2} V(y^{n}_{s})D_{r}y^{n}_{s}D_{r'}y^{n}_{s}+ \partial V(y^{n}_{s})D_{r'}D_{r}y^{n}_{s}.
\end{eqnarray*}
(The reader is referred to  \cite[equation (3.17)]{LLT2} for the  general  definition of the operator $\cl^{L}$.)
In both estimates   
\eqref{eqn.z.pbd1} and \eqref{eqn.dxiestimate2}, $K$ is a constant independent of $r$, $r'$, $j$ and $n$. 
 
 \end{theorem}

\section{Uniform Integrability  for Malliavin derivatives of the Euler scheme}\label{section.3}

In this section we     tackle the integrability issue for    the Malliavin derivatives of the Euler scheme. 
Before proceeding to our main considerations, some remarks about our global strategy are in order. Recall that $\cm_{0}$, $\cm_{1}$, $\cm_{2}$ are defined in \eqref{eqn.ms},   $b$ is an independent copy of $x$,   and $q^{b}$ are defined in \eqref{eqn.q}.  
{Recall that the signature $\bfx=S_{2}(x)=(x^{1}, x^{2})$ of $x$ is defined  in Definition \ref{def:rough-path}.} 

\begin{enumerate}[wide, labelwidth=!, labelindent=0pt, label=(\arabic*)]
\setlength\itemsep{.1in}
\item 
Due to the bound \eqref{eqn.z.pbd1}, the integrability of $\cm_{0}$, $\cm_{1}$, $\cm_{2}$ is our main task towards a uniform bound for the Malliavin derivatives  as a function of $n$. We mostly focus on this problem in the sequel.  

\item
Once the bounds in Theorem \ref{thm.xirr} are established, the fractional Brownian motion $b$ does not play any particular role in our estimates. Hence for notational sake we perform our computations below in the case  $b\equiv q^{b}\equiv 0$.   Modifications to cover the $b\neq 0$ case are trivial. 
\end{enumerate}
We now turn our attention to the integrability of the random variables $\cm_{i}$.

\subsection{Uniform Integrability of $\cm_{1}$ and $\cm_{2}$}\label{sec:unf-intg-M1-M2}

In this subsection, we consider the uniform integrability of $\cm_{1}$ and $\cm_{2}$ in \eref{eqn.ms}.  The proof is achieved thanks to a tail analysis of the cardinality of the large size steps, that is steps with size  $>\!\!>\al$. 

\begin{theorem}\label{thm.m2}
  Let $\cm_{2}=\cm_{2}(n)$  be 
  the random variable defined by \eqref{eqn.ms}. Specifically, since we are assuming that $b=0$, $\cm_{2}$ is given by 
  \begin{eqnarray}\label{e.m2}
\cm_{2} = \prod_{s_{j}\in S_{2}} 
\lp
K |\delta x_{s_{i}s_{i+1}}| +K \Delta^{2H}+1
\rp, 
\end{eqnarray}
where $S_{2}$ is the subset of $  \ll0,T\rr$ displayed in \eqref{eqn.cs2}.
Recall that $x$ is a fBm with Hurst parameter $H>1/3$. 
 Then we have   $\sup_{n\in \NN} \mE[ |\cm_{2}|^{\nu} ]<\infty$ for all $\nu\geq1$.   
\end{theorem}
\begin{proof}
Recall that   $\omega$   is defined by \eqref{eqn.control}. Since we assume $b=0$ and thus $q^{b}=0$, the control $\omega$ is reduced to  
\begin{align*}
\omega(t_{k},t_{k+1})=|\bfx_{t_{k}t_{k+1}}|^{ p} +|q_{t_{k}t_{k+1}}|^{p/2},  
\end{align*}
so that definition \eqref{eqn.cs2} yields the following relation:  
\begin{align}\label{eqn.s2-2}
S_{2}= \{t_{k}: \omega(t_{k},t_{k+1})>\al \} =& \{t_{k}: |\bfx_{t_{k}t_{k+1}}|^{ p} +|q_{t_{k}t_{k+1}}|^{p/2}>\al \} .
\end{align}
Moreover, it is readily checked that 
\begin{eqnarray*}
|\bfx_{t_{k}t_{k+1}}|^{p}+|q_{t_{k}t_{k+1}}|^{p/2}
\lesssim |\delta x_{t_{k}t_{k+1}}|^{p} + |x^{2}_{t_{k}t_{k+1}}| +|\Delta|^{2Hp}. 
\end{eqnarray*}
If we choose $n$ large enough   so that $|\Delta|^{2Hp}<\!\!< \al$, from the expression \eqref{eqn.s2-2} we get  
\begin{align*}
S_{2}
&\subset  \{  t_{k}: |\bfx_{t_{k}t_{k+1}}|^{ 2}  >K_{p}\al^{2/p} \}=:S_{3} . 
\end{align*}
This inclusion implies that 
\begin{align*}
 \cm_{2} =   &
 \prod_{s_{j}\in S_{2}  } 
  \lp K  |  \bfx_{s_{j}  s_{j+1}}|   +1  
\rp
\leq   
 \prod_{s_{j}\in S_{3}  } 
  \lp K  |  \bfx_{s_{j}  s_{j+1}}|   +1  
\rp
\leq   
C_{\al} \prod_{s_{j}\in S_{3}  } 
 K  |  \bfx_{s_{j}  s_{j+1}}|   
 ,
\end{align*}
where we have invoked the fact that $|\bfx_{s_{j}  s_{j+1}}|>K_{p}^{1/2}\al^{1/p}$ whenever $s_{j}\in S_{3}$ for the last inequality. 
We now divide the proof in several steps.


\noindent\emph{Step 1: Some pathwise bounds.}\quad 
Recall again that $x$ is a fBm with Hurst parameter $H>1/3$. Pick then  $\be<H$ such that \begin{eqnarray}\label{eqn.cg}
\cg\equiv\|\bfx\|_{[0,T],\be}  = \|\delta x\|_{[0,T], \be}+\|x^{2}\|_{[0,T], 2\be}^{1/2} \,, 
\end{eqnarray}
 as defined in \eqref{eq:def-norm-rp}, is almost surely finite. 
Then  since $|  \bfx_{s_{j}  s_{j+1}}|  \leq \cg n^{-\beta}$  we have 
\begin{align*}
\cm_{2} \leq   \prod_{t_{k}\in  S_{3}   }  
  K \lp  \cg n^{-\be} \rp\leq   ( 
  K \cg  n^{-\be})^{|S_{3}|}.
\end{align*}
In addition, according to \eqref{eqn.cg} we have 
\begin{align*}
S_{3}
\subset \bigcup_{i,j}
\lp
S^{i}_{31}\cup S^{ij}_{32}\cup S^{ij}_{33}
\rp
,
\end{align*}
where the sets $S_{31}$, $S_{32}$, $S_{33}$ are defined by 
\begin{align*}
S_{31}^{i} =  \{t_{k}: (  \delta x_{t_{k}t_{k+1}}^{i} )^{2}>\al_{p}\}, \quad S_{32 }^{ij}= \{t_{k}:    x_{t_{k}t_{k+1}}^{2,i,j} >\al_{p}\}, \quad S_{33 }^{ij}= \{t_{k}:    -x_{t_{k}t_{k+1}}^{2,i,j} >\al_{p}\},
\end{align*}
where $\al_{p}$ is some constant depending on $\al$. 
Therefore recalling that $K$ designates a generic constant,
we have obtained the following upper bound for the random variable $\cm_{2}$:  
\begin{align}\label{eqn.m2bd}
\cm_{2}\leq \lp\prod_{i=1}^{d} (K\cg n^{-\beta})^{|S^{i}_{31}|}  \rp \cdot
\lp\prod_{i,j=1}^{d} (K\cg n^{-\beta})^{|S^{ij}_{32}|}  \rp \cdot
\lp\prod_{i,j=1}^{d} (K\cg n^{-\beta})^{|S^{ij}_{33}|}  \rp .
\end{align}


In the following we   consider the integrability of the random variable  $(K\cg n^{-\beta})^{|S^{i}_{31}|}$ for all  $i=1,\dots, d$ which appear in the right-hand side of \eqref{eqn.m2bd}. 
The other terms in \eqref{eqn.m2bd} can be handled very similarly. 

\noindent\emph{Step 2: Tail estimates for $|S_{31}|$.}\quad 
For each subset $\mathbf{u}\equiv\{u_{j} ; j=1,\dots,n'\}$, $n'\leq n$ of the set of discrete instants  $ \{t_{k}, k=0,1,\dots,n\} $  we  denote 
\begin{align*}
\cx^{i} (\mathbf{u}) = \sum_{t_{k}\in \{u_{j};  \,j=1,\dots,n'\}}( | \delta x^{i}_{t_{k}t_{k+1}}|^{2} - \Delta^{ 2H} ) .
\end{align*}
Then notice that if $|S_{31}^{i}|=n'$, there exists a set $\mathbf{u}=\{u_{j};\, j=1,\dots, n'\}$ such that for all $j=1,\dots, n'$ we have $ (\delta x^{i}_{u_{j}u_{j+1}})^{2} >\al_{p} $. Hence if we take a constant  $K_{\al, p} $ such that  $K_{\al, p}<  \al_{p}$ and  take $n$ large enough, we have 
\begin{eqnarray*}
\cx^{i} (\mathbf{u})> n' K_{\al,p}. 
\end{eqnarray*}
We have thus proved that 
\begin{align*}
\{ |S_{31}^{i}|  = n'  \} &\subset  \bigcup_{\mathbf{u}\subset \{t_{k}\}} \lcl   \cx^{i} (\mathbf{u})  >K_{p,\al}\cdot n'\rcl. 
\end{align*}
 As a consequence of the above relation, we trivially get 
 \begin{eqnarray}\label{eqn.s31.sum}
\PP \{|S_{31}^{i}|    =n'  \} 
\leq&\, 
\sum_{\mathbf{u}\subset \{t_{k}\}}
\PP \left\{ \cx^{i}(\mathbf{u})>K_{p,\al}\cdot n' \right\} .
\end{eqnarray}

Next set   $\si_{\mathbf{u}}^{2}=(\mE|\cx^{i}(\mathbf{u}) |^{2}) $.  Owing to a slight variation of  \eqref{eqn.f.bd} we have  $\si_{\mathbf{u}}^{2}\lesssim  1 /n^{4H-1 }$.    Therefore starting from the right-hand side of \eqref{eqn.s31.sum} we get 
\begin{align}\label{eqn.s31.bd1}
\PP(|S^{i}_{31}| = n')\le&\,\sum_{\{u_{j}; \, j=1,\dots,n'\}\subset \{t_{k}\}} 
\PP \left\{ \frac{\cx^{i}(\mathbf{u})  }{\si_{\mathbf{u}}} >\frac{K_{p,\al}\cdot n' }{\si_{\mathbf{u}}}  \right\}
\nonumber
\\
\leq&\,\sum_{\{u_{j}; \,j=1,\dots,n'\}\subset \{t_{k}\}} 
\PP \left\{ \frac{\cx^{i} (\mathbf{u}) }{\si_{\mathbf{u}}} >\frac{K_{p,\al}\cdot n' }{1/n^{2H-1/2}}  \right\} . 
\end{align}
The right-hand side  of \eqref{eqn.s31.bd1} is handled in the following way: taking into account the fact that $\cx^{i}(\mathbf{u})/\si_{\mathbf{u}}$ is a normalized random variable in the second chaos of $x$, we apply Borell's inequality (see e.g. \cite[Theorem 5.12]{Hu}). In addition the number of sets of the form $\mathbf{u}=\{u_{j}; j=1,\dots, n'\}$ is ${n}\choose{n'}$. 
Hence we end up with 
 \begin{align}\label{eqn.s31.bd2}
\PP(|S^{i}_{31}| = n') \leq &\, \frac{n!}{n'!(n-n')!}\exp \lp - n^{2H-1/2}\cdot   K_{p,\al}\cdot n'   \rp
\nonumber
\\
\leq &\, n^{n'}\exp \lp - n^{2H-1/2}\cdot   K_{p,\al}\cdot n'   \rp. 
\end{align}
 
 \noindent\emph{Step 3: Computations involving $\cg$.}\quad 
Let us now turn our attention to the term $\cg$ in \eqref{eqn.cg}. Since our fBm $x$ is a Gaussian process, Fernique's lemma asserts that  $\PP(\cg >x)\leq e^{-K x^{2}}$ for a given constant $K$ and $x\geq 1$. This sub-Gaussian bound is sufficient to claim that for all  $n'\geq 1$ we have 
\begin{align}\label{eqn.g.bd}
\mE \cg ^{n'} \leq K (n')^{n' }= K e^{  n' \ln n'}. 
\end{align}
We are now ready to go back to the study of the random variable $(K \cg n^{-\be})^{\nu|S_{31}^{i}|}$, $K>0$. Namely we apply H\"older's inequality with two conjugates $p,q>1$, and we combine this with~\eqref{eqn.s31.bd2} and~\eqref{eqn.g.bd}. 
We get 
\begin{eqnarray}\label{eqn.fun.bd}
\mE \lp ( K \cg  n^{-\beta} )^{\nu n'} \mathbf{1}_{\{|S_{31}^{i}| =n'\}} \rp 
&=& K^{\nu n'} n^{-\be Kn'} (\mE [\cg^{pKn'}] )^{1/p} \PP (|S_{31}^{i}|=n')^{1/q} 
\nonumber
\\
&\leq&   e^{\nu  n' \ln n'} n^{-\beta Kn'}  \PP(|S_{31}^{i}|=n')^{1/q} 
\nonumber
\\
&\le& \exp \lp f(n')\rp, 
\end{eqnarray}
where the function $f$ is defined by 
\begin{eqnarray}\label{eqn.f.def}
f(n') :=    K_{1} n'\ln n'  -\be K n' \ln n +K_{2} n'\ln n - K_{3} n^{2H-1/2} n'    
,
\end{eqnarray}
for three positive constants $K_{1}$, $K_{2}$, $K_{3}$ whose exact value is irrelevant. 

We now compute the maximum of the function $f$ thanks to elementary considerations. First we calculate  
\begin{align*}
f''(n') = K/n'    \geq 0. 
\end{align*}
Therefore  $f$ is upward convex  and 
\begin{align}\label{eqn.supf}
\sup_{2\leq n'\leq n} f(n') \leq f(2)\vee f(n) \leq f(2)+ f(n)   . 
\end{align}
Moreover one can explicitly compute   $f(2)$ and $f(n)$ thanks to the expression \eqref{eqn.f.def}. We obtain 
\begin{align*}
f(n) = (K_{1}+K_{2}-\beta K) n\ln n - K n^{2H+1/2}, \qquad f(2) = K-K\ln n -K n^{2H-1/2}  . 
\end{align*}
Reporting this expression into \eqref{eqn.fun.bd}, we discover that 
\begin{align}\label{eqn.bd.s31}
  \mE ( K \cg  n^{-\beta} )^{\nu |S_{31}^{i}| } 
&= \sum_{n'=0}^{n} \mE \lp ( K \cg  n^{-\beta} )^{\nu n'} \mathbf{1}_{\{|S_{31}| =n'\}} \rp
\leq  \sum_{n'=2}^{n} \exp \lp f(n')\rp
\nonumber
\\
& \leq n \exp \lp f(n)+f(2)\rp
\leq n\exp ( C_{1}n\ln n - C_{2} n^{2H+1/2} ). 
\end{align} 

\noindent\emph{Step 4: Conclusion.}\quad 
Since we have assumed $H>1/3$, it is readily checked that the right-hand side of   \eqref{eqn.bd.s31} is dominated by a constant. Therefore, we end up with the inequality $\sup_{n\geq 1}\mE [( K \cg n^{-\be} )^{\nu |S_{31}^{i}|}]\equiv M<\infty$. 
This concludes the uniform (in $n$) integrability of $(K\cg n^{-\beta})^{\nu |S^{i}_{31}|}$, for all indices $i=1,\dots, d$.  The integrability of the other two quantities $  (K\cg n^{-\beta})^{\nu |S^{ij}_{32}|}  $ and $ (\cg n^{-\beta})^{\nu |S^{ij}_{33}|} $ can be shown in a similar way. Combining these integrability results with relation~\eref{eqn.m2bd} and with   H\"older's inequality, we obtain the uniform integrability of $\cm_{2}^{\nu }$.   Our proof is complete.
\end{proof}

Once the bound of $\cm_{2}$ is established, we can link the expected value of $\cm_{1}$ to that of $\cm_{2}$ by the observation that there are less   steps with a small size ($<\!\!< \al$)  than with a large size ($>\!\!>\al$). This is the content of the  following result. 

\begin{cor}\label{cor.m1}  Let $\cm_{1}$  be defined in \eref{eqn.ms} and we are still working with a fBm $x$ with Hurst parameter $H>1/3$. 
  Then $\sup_{n\in \NN}\mE[ |\cm_{1}|^{\nu }] <\infty$ for all $\nu\geq 1$. 
\end{cor}
\begin{proof}
In order to consider the integrability of $\cm_{1}$  we observe that by the definition of $S_{1}$, for $s_{j}\in S_{1}$ we have $\omega (s_{j+1} , s_{j+1}+\Delta )\geq \al/2$. Therefore,
\begin{align*}
 S_{1}\subset \{t_{k}: \omega (t_{k}, t_{k+1})>\al/2\}=:S_{3}',
\end{align*}
and so 
\begin{align*}
\cm_{1}= \prod_{s_{j}\in S_{1}} (K_{1}\omega(s_{j}, s_{j+1} )^{1/p}+1)
\leq    (K_{1}\al+1)^{|S_{1}|}\leq K^{|S_{3}'|}. 
\end{align*}
Observe that $K^{|S_{3}'|}$ is in the   form similar to $\cm_{2}$ in \eqref{eqn.ms}. 
So in a similar way as in Theorem~\ref{thm.m2}, we can show that $K^{\nu |S_{3}'|}$ and thus $\cm_{1}^{\nu}$ is uniformly integrable. 
\end{proof}

\subsection{Integrability of $\cm_{0}$}\label{sec:unf-intg-M0}

In this  section, we will take care of the products in \eqref{eqn.ms} involving small increments of $\omega$. Now recall that those increments, defined by \eqref{eqn.control}, involve the Gaussian process $\text{w}$ and the second chaos process $q$. The presence of $q$ will require a specific translation procedure on the Wiener space, which is carried out in Section \ref{section.tran}.  Then a weighted sum argument is invoked in Section \ref{section.int}.   
 

\subsubsection{Translation of the fBm and some functionals}\label{section.tran}
Let us recall that $x$ is a fBm with $H>1/3$ and $q$ is defined in \eref{eqn.q}. 
In this subsection, we consider an  upper-bound estimate for the translation of the fBm $x$ and the process $q$. Notice that in the sequel our generic random element in the space $\Omega$ will be denoted by $\phi$.
\begin{lemma}\label{lemma.tra}
  Take $3>p>1/H$ and   $p'>1$ such that  $1/p+1/p'>1$.   Let $h$ be a path in $C^{p'\text{-var}} ([0,T], \RR^{m})$, and let $T_{h}$ denote the translation operator: $T_{h}\phi = \phi+h$ on the Wiener space related to our fBm. Then the following translation inequality holds:
\begin{align}\label{eqn.translate.bd}
\| T_{h}q  \|_{p/2\tvr, \ll s,t\rr }^{p/2} +  \|T_{h}\bfx   \|_{p\tvr, \ll s,t\rr}^{p}  \leq 
K_{p} \lp
  \| q  \|_{p/2\tvr, \ll s,t\rr }^{p/2} +  \|\bfx   \|_{p\tvr, [s,t]}^{p}  +   \|h\|_{p'\tvr, [s,t]}^{p}
\rp,
\end{align}
where $K_{p}$ is a constant   depending only on $p$. 
\end{lemma}
\begin{proof}    The estimate of $ \|T_{h}\bfx   \|_{p\tvr, \ll s,t\rr}^{p}  $ is shown  in Lemma 3.1 \cite{CLL}. In the following we consider the estimate of $\| T_{h}q  \|_{p/2\tvr, \ll s,t\rr }^{p/2} $. 
Specifically, consider an element $u,v\in \cs_{2}(\ll s,t\rr)$. 
 By definition we can write
 \begin{align}\label{eqn.qh}
T_{h} q_{uv} =  q_{uv}+A^{1}_{uv} +A^{2}_{uv} +A^{3}_{uv} \,,
\end{align}
where
 \begin{align}
&A^{1}_{uv}=  \sum_{u\leq t_{k}<v} \int_{t_{k}}^{t_{k+1}} \delta h_{t_{k}r}\otimes dx_{r}\,,
\quad
A^{2}_{uv} = \sum_{u\leq t_{k}<v} \int_{t_{k}}^{t_{k+1}} \delta x_{t_{k}r} \otimes dh_{r}\,,
\label{enq.def.a1}
\\& A^{3}_{uv} = \sum_{u\leq t_{k}<v}  \int_{t_{k}}^{t_{k+1}} \delta h_{t_{k}r} \otimes dh_{r}\,. 
\label{enq.def.a3}
\end{align}
Next we further decompose the term $A^{1}$ into 
\begin{eqnarray}\label{eqn.a1}
A_{uv}^{1} = A_{uv}^{11} +A_{uv}^{12} ,
\end{eqnarray}
where $A^{11}$ and $A^{12}$ are respectively defined by 
\begin{align}\label{eqn.a11}
A^{11}_{uv} = \int_{u}^{v} \delta h_{u r} \otimes dx_{r} \, ,
\quad\text{and}\quad
A^{12}_{uv}= - \cj_{u}^{v} ( h, x )  , 
\end{align}
and where the term $\cj$ above is given as   \begin{eqnarray}\label{eqn.cj}
\cj_{u}^{v} ( h, x )  =\sum_{u\leq t_{k}<v} \delta h_{ut_{k}} \otimes  \delta x_{t_{k}t_{k+1}} \, .
\end{eqnarray}
 In the following, we bound the terms on the right side of  \eref{eqn.qh}. 

First,   by a direct computation for all $(u,r,v)\in \cs_{3} (\ll s,t \rr)$ we have 
\begin{align}\label{eqn.da12}
\delta A^{12}_{urv} = \delta h_{ur}\otimes \delta  x_{rv}.  
\end{align}
In order to bound $\delta A^{12}$, we   consider the   function 
\begin{eqnarray}\label{eqn.control2}
\omega (u,v) =: \|h\|_{p'\text{-var}, [u,v]} \|\bfx   \|_{p\tvr, [u,v]}. 
\end{eqnarray}
 It is well known that since $1/p+1/p'>1$, $\omega$ is a control function. In fact, it is easy to show that $\omega_{1} =:\omega^{1/\mu}$ is a control function for   $\mu$ such that $1/p+1/p'>\mu> 1$.   
It  follows from~\eref{eqn.da12} and the definition of $\omega_{1}$ that 
\begin{align*}
| \delta A^{12}_{urv} |  \leq  \omega_{1}(u,v)^{\mu}. 
\end{align*}
In addition, it is readily checked from our definition \eqref{eqn.cj} that $  A^{12}_{t_{k}t_{k+1}}=0$ for all $t_{k}\in \ll s,t\rr$. Therefore a direct application of   Lemma \ref{lem2.4} yields:
\begin{align}\label{eqn.a12}
|A^{12}_{uv}|\leq K_{\mu} \, \omega_{1}(u,v)^{\mu}=K_{\mu} \, \omega(u,v). 
\end{align}

Let us turn to the estimate of  $A^{11}$ defined by \eqref{eqn.a11}. In that case, due to the fact that $A^{11}$ can be interpreted as a Young integral, some elementary estimates (see e.g. \cite{Young}) reveal that    
\begin{eqnarray}\label{eqn.a11.bd}
|A^{11}_{uv}  |\leq  \omega(u,v).
\end{eqnarray}
 Hence reporting    \eref{eqn.a12} and \eqref{eqn.a11.bd}   into \eref{eqn.a1} we end up with  
 \begin{align}\label{eqn.a1.bd}
|A^{1}_{uv}  |\leq (K_{\mu}+1) \omega(u,v)\leq (K_{\mu}+1) (\|h\|_{p'\text{-var}, [u,v]}^{2}+ \|\bfx   \|_{p\tvr, [u,v]}^{2}),  
\end{align}
where we recall that the control $\omega$ is given by \eqref{eqn.control2}. 
The term $A^{2}$ in \eqref{enq.def.a1} can be bounded in a similar way as for $A^{1}$, and 
we obtain the same estimate as in \eqref{eqn.a1.bd}.  The details are thus omitted. 

In order to bound $A^{3}$ defined by \eqref{enq.def.a3}, we apply Young's inequality again and also the super-additivity of the control $\omega_{2}(u,v)=:\|h\|_{p'\tvr, [u,v]}^{2}$ (Notice that $\omega_{2}$ is a control owing to the fact that $p'<2$). We get
\begin{align*}
  |A^{3}_{uv} | \leq \sum_{u\leq t_{k}<v} \|h\|_{p'\tvr, [t_{k}, t_{k+1}]}^{2}\leq \|h\|_{p'\tvr, [u,v]}^{2}. 
\end{align*}
 
 Putting together the estimates of $A^{1}$, $A^{2}$ and $A^{3}$ and   equation \eref{eqn.qh}, we obtain
 \begin{align}\label{e.tq}
|T_{h}q_{uv}| \leq |q_{uv}|  +2(K_{\mu}+1) \|\bfx   \|_{p\tvr, [u,v]}^{2}
 + (2K_{\mu}+3)\|h\|_{p'\tvr, [u,v]}^{2}.
\end{align}
Now consider a generic partition $\pi=\{u_{j}\}$ of $  \ll s,t \rr$. Thanks to \eqref{e.tq} and super-additivity properties we have 
\begin{align*}
\sum_{\{u_{j}\}} |T_{h}q_{u_{j}u_{j+1}}|^{p/2}  
&\leq K_{p} \lp
\sum_{\{u_{j}\}} | q_{u_{j}u_{j+1}}|^{p/2} +\sum_{\{u_{j}\}} \|\bfx   \|_{p\tvr, [u_{j},u_{j+1}]}^{p}  + \sum_{\{u_{j}\}} \|h\|_{p'\tvr, [u_{j},u_{j+1}]}^{p}
\rp
\\
&\leq 
K_{p} \lp
  \| q  \|_{p/2\tvr, [s,t] }^{p/2} +  \|\bfx   \|_{p\tvr, [s,t]}^{p}  +   \|h\|_{p'\tvr, [s,t]}^{p}
\rp.
\end{align*}
Finally, taking the sup over all partitions of $[s,t]$
 on the left side we obtain the desired estimate \eqref{eqn.translate.bd}. 
\end{proof}

\subsubsection{Integrability of $\cm_{0}$}\label{section.int}
This section is devoted to a study of the intermediate sized increments of $\omega$. Otherwise stated, 
we are ready to show the uniform integrability of $\cm_{0}$.

\begin{theorem}\label{thm.m0}
  Let $S_{0}$ and $\cm_{0}$  be defined in \eref{eqn.ms}, for a fBm $x$ with Hurst parameter $H>1/3$ and a threshold $\al>0$.
  Then for any given $\ga<2H+1$ there exists $K=K_{\ga}$ such that for all $a\geq 1$ we have  
 \begin{align}\label{eqn.S0tail1}
\PP(|S_{0}|>a)\leq 
K e^{-Ka^{\ga}} .
\end{align}
  In particular, $\sup_{n\in \NN} \mE [\cm_{0}^{\nu}] <\infty$ for all $\nu\ge 1$.   
\end{theorem}

\begin{proof}
The proof will be done in several steps.

\noindent {\it Step 1. Preparations.}
Let us go back to inequality \eqref{eqn.translate.bd}. Remember that $p>1/H$ therein. Since $H>1/4$, it is easily checked that one can pick  
  $p'>(H+1/2)^{-1}$ such that $p$, $p'$ still satisfy $\frac{1}{p}+\frac{1}{p'}>1$. This pair of $p$, $p'$ will be fixed for the remainder of the proof. Recalling the constant $K_{p}$ featuring in \eqref{eqn.translate.bd} and our threshold $\al$, we also choose 
   $\be>0$ small enough so that   $\al/2-K_{p}\be>0$. Since $p>1/H$, according to   
\cite[Remark 3.6]{LT} there exists an almost surely  finite random variable $G_{p}$ such that $\sup_{n\in\NN}\| q  \|_{p/2\tvr, \ll s,t\rr }^{p/2} \leq G_{p}$.  
Related to those quantities, we define the following two sets:   
\begin{eqnarray*}
A^{n}&=& \{\phi\in \Omega:   \| q  \|_{p/2\tvr, \ll s,t\rr }^{p/2} +  \|\bfx   \|_{p\tvr, [s,t]}^{p}  <\be \} ,
\\
 A&=& \{\phi\in \Omega:  G_{p}+    \|\bfx   \|_{p\tvr, [s,t]}^{p}  <\be \}, 
\end{eqnarray*}
where we recall that the typical element of $(\Omega, \cf, \PP)$ is denoted by $\phi$. 
It is clear that $A\subset A^{n}$. 

\noindent {\it Step 2. Tail inclusion relations.} 
Let $a\geq 1$ be our generic threshold. Having the notation of Step 1 in mind we define a constant   $\kappa$ as follows: 
\begin{eqnarray}\label{eqn.kappa}
\kappa= \lp\frac{\al/2-K_{p}\be}{K_{p}} \rp^{1/p} a^{1/p'}.
\end{eqnarray}
Recall that the Cameron-Martin type space  $\bar{\ch}$ is defined in Section \ref{subsection.d}. 
Let us   call $B_{\bar{\ch}}$   the unit ball in $\bar{\ch}$, namely: $B_{\bar{\ch}} = \{h\in \bar{\ch}; \|h\|_{\bar{\ch}}\leq 1\}$.
Our first aim is to show that 
\begin{align}\label{eqn.S0tail}
A +\kappa B_{\bar{\ch}} \subset   A^{n}+\kappa B_{\bar{\ch}}\subset  \{|S_{0}|\leq a\} .
\end{align}

Suppose that $\phi \in  A^{n}+\kappa B_{\bar{\ch}}$. In the following, we  show that $|S_{0}|\leq a$ for such $\phi$, which then implies the relation \eref{eqn.S0tail}.      First, for $\phi \in  A^{n}+\kappa B_{\bar{\ch}}$ we have $\phi - h \in A^{n}$ for some $h\in \kappa B_{\bar{\ch}}$, and thus 
\begin{align*}
 \| q (\phi - h ) \|_{p/2\tvr, \ll s,t\rr }^{p/2} +  \|\bfx (\phi - h )  \|_{p\tvr, [s,t]}^{p}  <\be. 
\end{align*}
Recall that $T_{h}\bfx(\phi)= \bfx(\phi+h) $ for any $h\in \bar{\ch}$ almost surely. Hence the above relation becomes
\begin{eqnarray}
\label{eqn.Thbeta}
 \| T_{-h}q (\phi   ) \|_{p/2\tvr, \ll s,t\rr }^{p/2} +  \|T_{-h}\bfx (\phi  )  \|_{p\tvr, [s,t]}^{p}  <\be.
\end{eqnarray}
We now consider the control $\omega$ defined by $\omega (s,t) = \|q\|_{p/2\text{-var}, \ll s,t\rr}^{p/2} + \|\bfx\|_{p\text{-var}, \ll s,t \rr}^{p}$. For a generic element $\phi \in  A^{n}+\kappa B_{\bar{\ch}}$ we have 
\begin{eqnarray*}
\omega (s,t) (\phi)&=&\| T_{h}T_{-h} q (\phi   ) \|_{p/2\tvr, \ll s,t\rr }^{p/2} +  \| T_{h}T_{-h} \bfx (\phi  )  \|_{p\tvr, \ll s,t\rr}^{p}.
\end{eqnarray*}
Hence invoking Lemma \ref{lemma.tra} we get 
\begin{eqnarray*}
\omega (s,t) (\phi) &\leq&
K_{p}\lp
\| T_{-h}q (\phi   ) \|_{p/2\tvr, \ll s,t\rr }^{p/2} +  \|T_{-h}\bfx (\phi  )  \|_{p\tvr, [s,t]}^{p}+ \|h\|^{p}_{p'\tvr, [s,t]}
\rp  ,
\end{eqnarray*}
and owing to \eref{eqn.Thbeta}  one ends up with the following relation valid for all $\phi \in  A^{n}+\kappa B_{\bar{\ch}}$:
\begin{eqnarray*}
\omega (s,t) (\phi) &\leq& K_{p}\be+K_{p} \|h\|^{p}_{p'\tvr, [s,t]}.
\end{eqnarray*}
In particular, when $s=s_{j}$ and $t=s_{j+1}$ for $s_{j}\in S_{0}$ we obtain
\begin{align*}
\al/2 \leq \omega (s_{j}, s_{j+1})\leq K_{p}\be+K_{p} \|h\|^{p}_{p'\tvr, [s_{j}, s_{j+1}]},
\end{align*}
and thus
\begin{align}\label{eqn.hsj}
   \|h\|^{p'}_{p'\tvr, [s_{j}, s_{j+1}]}
 \geq \lp\frac{\al/2-K_{p}\be}{K_{p}} \rp^{p'/p}.
\end{align}
Since $\omega_{1}(s,t)\equiv \|h\|_{p', [s,t]}^{p'}$ is a control it follows  from \eref{eqn.hsj} that
\begin{eqnarray}\label{eqn.s0.bd}
 \|h\|^{p'}_{p'\tvr, [0,T]}
 \geq  \sum_{s_{j}\in S_{0}} \|h\|^{p'}_{p'\tvr, [s_{j}, s_{j+1}]}
 \geq \lp\frac{\al/2-K_{p}\be}{K_{p}} \rp^{p'/p} |S_{0}| .
\end{eqnarray}

We now specify the left-hand side of \eqref{eqn.s0.bd}. First since we have chosen $p'> (H+1/2)^{-1}$, the reference \cite[Page 14]{CHLT} asserts that  $|h|_{\bar{\ch}} \geq   \|h\| _{p'\tvr, [0,T]}$. Moreover we have assumed that $h\in \kappa B_{\bar{\ch}}$. We thus obtain 
\begin{align*}
 \kappa^{p'}\geq  |h|_{\bar{\ch}}^{p'}\geq   \|h\|^{p'}_{p'\tvr, [0,T]}.  
\end{align*}
Plugging this inequality into \eqref{eqn.s0.bd}, we obtain that if $\phi \in A^{n}+\kappa B_{\bar{\ch}}$ then 
\begin{align*}
    |S_{0}|
    \leq  \kappa^{p'} \lp\frac{\al/2-K_{p}\be}{K_{p}} \rp^{-p'/p} =a  ,
\end{align*}
where the last identity stems from the definition \eqref{eqn.kappa} of $\kappa$. We have thus proved that if $\phi \in A^{n}+\kappa B_{\bar{\ch}}$, then $|S_{0}|\leq a$. This   concludes the proof of \eref{eqn.S0tail}. 

\noindent {\it Step 3. Tail estimates.} \quad
Let us introduce some extra bits of notation. Namely we write $\Phi$ for the standard Gaussian CDF. For a set $A\subset \Omega$ we also define $a_{A}\in \RR$ as the number such that  $\Phi (a_{A}) = \PP(A)$. Then the isoperimetric type inequality in \cite[Theorem 4.3]{Ledoux},
 together with \eref{eqn.S0tail}, yield 
\begin{align*}
\PP(|S_{0}|>a)\leq \PP((A+\kappa B_{\bar{\ch}})^{c}) \leq e^{-K (a_{A}+\kappa )^{2}}=e^{K\frac{\kappa^{2}}{2}-K (a_{A}+\kappa )^{2}} e^{-K\frac{\kappa ^{2}}{2}} .
\end{align*}
Let   $K_{A} >0$ be   an upper bound of the quadratic function $f(\kappa ) =K\frac{\kappa^{2}}{2}-K (a_{A}+\kappa )^{2} $ on $\RR$. Then considering a constant $K$ which can change from line to line  and recalling the definition~\eqref{eqn.kappa} of $\kappa$, we get
\begin{align}\label{eqn.S0-tail-bd}
\PP(|S_{0}|>a)\leq K_{A} \, e^{-K\frac{\kappa^{2}}{2}}=K_{A}e^{-Ka^{2/p'}}. 
\end{align}
Recall again 
 that    $p'$ can be chosen arbitrarily close to $(H+1/2)^{-1}$. Hence $2/p'$ is of the form $2H+1-\ep$ for a small $\ep >0$. 
 This conclude the tail estimate \eref{eqn.S0tail1}. 
 It follows immediately from   \eref{eqn.S0tail1}   that $|S_{0}|^{\nu}$ and thus $\cm_{0}^{\nu}$ is uniformly integrable for any $  \nu\ge 1$. 
\end{proof}

\subsection{Integrability of Malliavin derivatives}

With the preliminary results of Sections \ref{sec:unf-intg-M1-M2} and \ref{sec:unf-intg-M0} in hand, we can now turn to the integrability result for the Malliavin derivatives of the Euler scheme. Notice that we restrict our analysis here to the first 2 Malliavin derivatives of $y^{n}$. However, it is clear that our estimates could be extended to arbitrary Malliavin derivatives.

\begin{thm}\label{thm.ixi}

Let $y^{n}$ be the Euler scheme defined by \eqref{e4}. The first and second Malliavin derivatives of $y^{n}$ are contained in the vector $\xi^{n}$ introduced in \eqref{d1}. We assume that the vector field $V$ is $C^{4}_{b}$ and that $x$ is a fBm with Hurst parameter $H>1/3$. Then for all $\nu\geq1$ we have 
\begin{equation}\label{b1}
\mE\lc \|\xi^{n}\|_{p\text{-var}}^{\nu} \rc <\infty. 
\end{equation}
In particular, the following sup-norm inequality holds true:
\begin{equation}\label{b2}
\mE\Big[ \sup_{n\in\NN,\, r, r',t\in [0,T]} |  \xi^{n}_{t}  |^{\nu}\Big] <\infty. 
\end{equation}

\end{thm}
 \begin{proof}
Inequality \eqref{b1} follows by showing that all   terms  in  the right-hand side of \eqref{eqn.z.pbd1} have moments of all orders. 
 Applying 
  Theorem \ref{thm.m0}, 
Corollary \ref{cor.m1} and Theorem \ref{thm.m2} respectively  we obtain the integrability of  $\cm_{0}$, $\cm_{1}$, $\cm_{2}$. The integrability of $|S_{0}| $ follows from \eqref{eqn.S0tail1}. The integrability   of $|S_{1}| $ and $|S_{2}| $  are implied by the relation $|S_{i}|\lesssim \cm_{i}   $, $i=1,2$, respectively.   The upper bound~\eqref{b2} is an easy consequence of~\eqref{b1}.
\end{proof}

\section{Weak convergence}\label{section.4}

With our   Malliavin derivative and integrability estimates in hand,
in this section we consider the weak convergence of the Euler scheme. The first sections are preparations of the main result.

\subsection{Estimation of an inner product in $\ch^{\otimes 2}$}\label{section4.1}

In this subsection, we derive a useful upper-bound estimate for an inner product of the form  $\langle \varphi, \mathbf{1}_{[u,v]} \otimes  \mathbf{1}_{[s,t]}     \rangle_{\ch^{\otimes 2}}$, involving some indicator functions. We first need a positivity result for the rectangular increment function $R$ of the fBm. 

\begin{lemma}\label{lemma.drp}
Recall that the covariance $R$ is defined in \eqref{eq:cov-fbm}, with rectangular increments $R ([u,v], [s,t])  $ introduced in \eqref{eq3.1}. Then 
for any $u,v,s,t\in \RR$   such that $ s\leq u\leq v\leq  t $     we have 
\begin{eqnarray*}
R([u,v], [s,t]) \geq 0. 
\end{eqnarray*}

\end{lemma}
\begin{proof}
We first write
\begin{eqnarray*}
R([u,v], [s,t])  = R([u,v], [u,v]) +R([u,v], [u,v]^{C}) ,
\end{eqnarray*} 
where we denoted $[u,v]^{C} = [s,t]\setminus[u,v]$.  Since $R([u,v], [u,v]) = (v-u)^{2H}$ it suffices to show that $R([u,v], [u,v]^{C}) \geq -(v-u)^{2H}$.

By definition \eqref{eq:cov-fbm}-\eqref{eq3.1} of $R$ we can write
\begin{eqnarray}\label{eqn.rc}
&&R([u,v], [u,v]^{C}) = R([u,v], [s,u] )+ R([u,v], [v,t] )
\nonumber
\\
&=& \frac12 ( |v-s|^{2H}-|u-s|^{2H}-|v-u|^{2H} ) + \frac12 ( |t-u|^{2H}-|t-v|^{2H}-|v-u|^{2H} )  .  
\end{eqnarray}
Note that  $|v-s|^{2H}-|u-s|^{2H}$ and  $|t-u|^{2H}-|t-v|^{2H}$  are nonnegative. We thus obtain 
\begin{eqnarray}\label{eqn.rc2}
R([u,v], [u,v]^{C})  \geq -|v-u|^{2H}. 
\end{eqnarray}
The proof is complete.
\end{proof}

The above positivity result leads to a   surprisingly easy  bound on products in $\ch^{\otimes 2}$.  
\begin{lemma}\label{lemma.bdd2}
Let $\varphi\in\ch^{\otimes 2}$.  For each $s\in [0,T]$ we assume that $\varphi(s,\cdot)\in C^{p\text{-var}}([0,T])$. Let $f_{s}= \|\varphi(s,\cdot)\|_{p\text{-var}}$ for $s\in [0,T]$. We also  assume that  $f\in C^{p\text{-var}}([0,T])$.   For $s, t, u,v\in[0,T]$: $s<t$, $u<v$   we define 
    $\al(\eta, \zeta) =  \mathbf{1}_{[u,v]} (\eta )  \mathbf{1}_{[s,t]} (  \zeta) $ for $\eta, \zeta \in [0,T]$. 
Then the following relation holds
  \begin{eqnarray}\label{eqn.phialpha}
\Big|
\langle \varphi, 
\al
 \rangle_{\ch^{\otimes 2}}
\Big|
   \leq  4(t-s)^{ 2H} (v-u)^{2H} \|\varphi\|_{\infty} . 
\end{eqnarray}
\end{lemma}
\begin{proof}
Starting from   Definition \ref{def3.2} and taking limits on indicator functions of rectangles (similarly to \cite[Lemma 15.39]{FV}), one can prove that 
 the norm in $\ch^{\otimes 2} $ can be expressed as a double 2D Young integral of the form  
\begin{eqnarray*}
\langle \varphi, \al \rangle_{\ch^{\otimes 2}} =
\langle \varphi, \mathbf{1}_{[u,v]} \otimes  \mathbf{1}_{[s,t]}  \rangle_{\ch^{\otimes 2}} = 
\int_{[0,T]^{4}}  \varphi(\eta, \zeta) \,  \mathbf{1}_{[u,v]}    (\eta' ) \mathbf{1}_{[s,t]}  (  \zeta') \, dR(\eta, \eta')dR(\zeta, \zeta') . 
\end{eqnarray*}
One can then integrate out the $\eta'$ and $\zeta'$ variables in order to get 
\begin{equation}\label{f1}
\langle \varphi, \al \rangle_{\ch^{\otimes 2}}= 
\int_{[0,T]^{2}}  
 \varphi(\eta, \zeta) \, dR(\eta, [u, v]) dR(\zeta, [s, t]) .
\end{equation}
We further decompose the inner product $\langle \varphi, \al \rangle_{\ch^{\otimes 2}}$ using the identity:
\begin{eqnarray*}
\1 = \al_{1}+\al_{2}+\al_{3}+\al_{4},
\end{eqnarray*}
where the functions $\al_{1},\ldots,\al_{4}$ are given by
\begin{eqnarray*}
\al_{1}(\eta, \zeta) = \mathbf{1}_{[u,v]}    (\eta  ) \mathbf{1}_{[s,t]}  (  \zeta ), 
\qquad
\al_{2} (\eta, \zeta)= \mathbf{1}_{[u,v]^{C}}    (\eta  ) \mathbf{1}_{[s,t]}  (  \zeta ), 
\\
\al_{3}(\eta, \zeta) = \mathbf{1}_{[u,v]}    (\eta  ) \mathbf{1}_{[s,t]^{C}}  (  \zeta ), 
\qquad
\al_{4} (\eta, \zeta)= \mathbf{1}_{[u,v]^{C}}    (\eta  ) \mathbf{1}_{[s,t]^{C}}  (  \zeta )
,  
\end{eqnarray*}
and where similarly to what we wrote in Lemma \ref{lemma.drp}, we have set 
$[s,t]^{C} = [0,T]\setminus [s,t]$ and $[u,v]^{C} = [0,T]\setminus [u,v]$.  
Otherwise stated,  we recast \eqref{f1} as
 \begin{eqnarray}
\label{eqn5.1}
\langle \varphi, \al \rangle_{\ch^{\otimes 2}}= 
\int_{[0,T]^{2}}  
 \varphi(\eta, \zeta) \, dR(\eta, [u, v]) dR(\zeta, [s, t]) = 
 \sum_{i=1}^{4} J_{T}^{i},
\end{eqnarray}
where the terms $J_{T}^{i}$ are respectively defined by 
\begin{eqnarray*} 
J_{T}^{i} = \int_{[0,T]^{2}}  
 \varphi(\eta, \zeta) 
\al_{i}(\eta, \zeta)
  \, dR(\eta, [u, v]) dR(\zeta, [s, t]) 
  . 
\end{eqnarray*}
Those four terms will be handled with slightly different arguments. 
 That is for $J^{1}_{T}$, owing to   Lemma \ref{lemma.drp} we have that  both  $dR(\zeta, [s, t])$ and $dR(\eta, [u, v]) $ are positive when $ \eta \in [u,v] $ and $ \zeta \in [s,t] $. Therefore, we have \begin{eqnarray}\label{eqn.varphi.abs}
 |J^{1}_{T}| \leq \|\varphi\|_{\infty} R([s,t], [s, t])R([u,v], [u,v])
  \leq  (t-s)^{ 2H} (v-u)^{2H} \|\varphi\|_{\infty} .  
\end{eqnarray}

For the  second term $J_{T}^{2}$ in \eqref{eqn5.1}  
we observe that $ dR(\zeta, [s, t]) $  is positive and $dR(\eta, [u, v])$ is negative. Therefore the product $dR(\eta, [s,t])\cdot dR(\eta, [u, v])$ does not change sign and we get  
\begin{eqnarray*}
 |J^{2}_{T}|   \leq    \|\varphi\|_{\infty} 
 | R([u,v], [u,v]^{C}) R([s,t], [s,t] ) |. 
\end{eqnarray*}
Hence thanks to an elementary computation similar to \eqref{eqn.rc}-\eqref{eqn.rc2} we discover that  
\begin{eqnarray}\label{eqn.j2.bd}
 |J^{2}_{T}|   \leq  (t-s)^{ 2H} (v-u)^{2H} \|\varphi\|_{\infty} . 
\end{eqnarray}
In conclusion, gathering \eqref{eqn.varphi.abs}, \eqref{eqn.j2.bd} and  similar bounds for $J^{3}_{T}$, $J^{4}_{T}$ into \eqref{eqn5.1}, we get the desired estimate \eqref{eqn.phialpha}.   This  concludes the proof. 
\end{proof}

We now extend the previous lemma to the indicator of a simplex in $[0,T]^{2}$.

\begin{lemma}\label{lemma.d2f}
Let $\varphi \in   \ch^{\otimes 2} $ be as in Lemma \ref{lemma.bdd2}. Let $\be\in \ch^{\otimes 2}$ be of the form 
 \begin{eqnarray}\label{eqn.beta}
\beta_{st} (u,v)= \mathbf{1}_{\cs_{2}([s,t])}(u,v)  ,
\end{eqnarray}
where we recall that the simplex $\cs_{2}([s,t])$ is defined in Notation~\ref{general-notation}.
  Then 
  there exists a constant $C_{H} $ such that 
  the following relation holds
  \begin{eqnarray}\label{eqn.phibeta}
\Big|
\langle \varphi, \beta_{st}  \rangle_{\ch^{\otimes 2}}
\Big|
   \leq C_{H} (t-s)^{ 4H}  \|\varphi\|_{\infty} . 
\end{eqnarray}
\end{lemma}
\begin{proof}
We will use a dyadic partition of the function $\be$. Namely for $n\geq 0$ and $0\leq i \leq 2^{n}$ we set 
  $u_{i,n} = s + 2^{-n}(t-s)i$. 
Next for $\ell\geq 1$ we define  
\begin{eqnarray*}
\be_{st}^{\ell}  =
\sum_{n=1}^{\ell} \sum_{i=0}^{2^{n-1}-1}   \mathbf{1}_{[u_{2i,n}, u_{2i+1, n }]\times  [u_{2i+1, n}, u_{2i+2, n}]  }  . 
\end{eqnarray*}
Then it  can be shown that $\|\be^{\ell}_{st} - \be_{st}\|_{\ch^{\otimes 2}}\to 0 $. In order to prove the lemma it thus suffices to show that for all $\ell\geq 1$ we have 
 \begin{eqnarray*}
\Big|
\langle \varphi, \beta_{st}^{\ell}  \rangle_{\ch^{\otimes 2}}
\Big|
   \leq (2^{4H}-2)^{-1} (t-s)^{ 4H}  \|\varphi\|_{\infty} .   
\end{eqnarray*}
  In the following we prove this relation with the help of Lemma \ref{lemma.bdd2}.  We first observe that by the definition of $\be^{\ell}$ 
  \begin{eqnarray*}
\Big|
\langle \varphi, \beta_{st}^{\ell}  \rangle_{\ch^{\otimes 2}}
\Big|
   \leq  \sum_{n=1}^{\ell} \sum_{i=0}^{2^{n-1}-1}   
   \Big|
\langle \varphi,  \mathbf{1}_{[u_{2i,n}, u_{2i+1, n }]\times  [u_{2i+1, n}, u_{2i+2, n}]  }   \rangle_{\ch^{\otimes 2}}
\Big|.
\end{eqnarray*}
Applying Lemma \ref{lemma.bdd2} with $(s, u,v, t) = (u_{2i,n}, u_{2i+1, n }, u_{2i+1, n}, u_{2i+2, n})$, we obtain
 \begin{eqnarray*}
\Big|
\langle \varphi, \beta_{st}^{\ell}  \rangle_{\ch^{\otimes 2}}
\Big|
   \leq  \sum_{n=1}^{\ell} \sum_{i=0}^{2^{n-1}-1}   
   (2^{-n} (t-s))^{2H} (2^{-n}(t-s))^{2H}  \|\varphi\|_{\infty}
  \\
  =
  \frac12 (t-s)^{4H} \|\varphi\|_{\infty} \sum_{n=1}^{\ell} (2^{n})^{1-4H}  
  \leq 
  \frac{1}{2^{4H}-2}  (t-s)^{4H} \|\varphi\|_{\infty}   .
\end{eqnarray*}
This completes the proof of our claim \eqref{eqn.phibeta}.
\end{proof}

In the sequel we will also need an inequality for products in $\ch$. Its proof is similar to  the proof of Lemma \ref{lemma.d2f} and is omitted for sake of conciseness.   
\begin{lemma}\label{lemma.bdd1}
Let $\varphi \in \ch$ be a function in  $      C^{p\text{-var}} ([0,T])$. 
 Then the following relation holds
  \begin{eqnarray*}
\Big|
\langle \varphi, 
\mathbf{1}_{[s,t]}
 \rangle_{\ch }
\Big|
   \leq  (t-s)^{ 2H}  \|\varphi\|_{\infty}   
\end{eqnarray*}
for all $(s,t)\in\cs_{2}(0,T)$. 
\end{lemma}

\subsection{An extension of the sewing lemma}\label{section4.2}
 In this section we extend Lemma \ref{lem2.4} to the integral of two controlled processes. Our findings are summarized in the following lemma.  
 \begin{lemma}\label{lem.sew.discrete}
 {Let $S_{2}(x):=(x^{1}, x^{2})$ be the geometric rough path above $x$ as given in Definition~\ref{def:rough-path}, and $p<3$.} 
We consider two couples of paths  
  $ (z,z') $ and $ (\tilde{z},\tilde{z}') $ with    $z,\tilde{z}\in C([ s,t], \RR^{m})$ and $z',\tilde{z}'\in C([ s,t], \RR^{m\times d})$. Let $\omega^{x}(u,v) = \|x\|_{p\text{-var}, [u,v]}^{p}$ for $(u,v)\in \cs_{2}([s,t])$.  
  We assume the existence of two controlled functions  $\omega^{z} $, $\omega_{1}^{z}$ and $\omega^{z'} $ on  $\ll s,t\rr $ such that for all $(u,v)\in \cs_{2}(\ll s,t\rr) $  we have     
\begin{align}\label{eqn.sew.condition}
|
\delta z_{uv}- z_{u}'   x^{1}_{uv} | \leq \omega^{z}(u,v)^{2/p} , 
\qquad
|\delta z_{uv}| \leq \omega_{1}^{z} (u,v)^{1/p}, 
\qquad |\delta z_{uv}'|\leq  \omega^{z'}(u,v)^{1/p}  .
\end{align}
We also assume that the relations in \eqref{eqn.sew.condition} hold for $\tilde{z}$, with related increments $\tilde{z}',\omega^{z'}$. 
Next we introduce some new control functions: 
    \begin{eqnarray}\label{eqn.omega.xzz}
\omega^{x,z ,z '}  = \omega^{x }+\omega^{z  }+\omega^{z ' }, \qquad\text{and}
\qquad \omega^{x,z ,z '}_{1}  = \omega^{x,z ,z '} +\omega_{1}^{z} \, ,
\end{eqnarray}
 and similarly for $(\tilde{z}, \tilde{z}')$. We now define some remainder terms in the integrals of $z$ with respect to $\tilde{z}$ or $x$. Namely for $(u,v)\in \cs_{2}(\ll s,t \rr)$
 we set 
 \begin{eqnarray}\label{f2}
R_{uv}^{z\tilde{z}} = \int_{u}^{v} ( \delta z_{ur} - z_{u}'   x^{1}_{ur}  ) \otimes d \tilde{z}_{r} 
\qquad\text{and}\qquad
R_{uv}^{ \tilde{z}x} = \int_{u}^{v} ( \delta \tilde{z}_{ur} - \tilde{z}_{u}'   x^{1}_{ur}  ) \otimes d x_{r} , 
\end{eqnarray}
where the above integrals are understood in the rough path sense. 
We suppose   that the increments $R$ are such that for any point $t_{k}$ in our generic partition of $[s,t]$ we have 
\begin{eqnarray}\label{f3}
|R^{z\tilde{z}}_{t_{k}t_{k+1}}| \leq \omega^{x}(t_{k}, t_{k+1})^{3/p} \,
\qquad\text{and}\qquad
|R^{ \tilde{z}x}_{t_{k}t_{k+1}}| \leq \omega^{x}(t_{k}, t_{k+1})^{3/p}. 
\end{eqnarray}
Then the following relation holds for all $(u,v)\in \cs_{2}(\ll s,t\rr)$:
\begin{eqnarray}\label{eqn.rbd}
\Big|
R^{z\tilde{z}}_{uv}  
 \Big|
   \leq  K_{p} [\omega^{R}(u,v) ] ^{\mu} ,
\end{eqnarray}
where $\mu>1$ is a given constant and
where $K_{p}>0$ is a   constant depending on $p$. In \eqref{eqn.rbd}, the control $\omega^{R}$ is also defined by the relation
\begin{multline}\label{eqn.wr}
[\omega^{R}(u,v) ] ^{\mu} :=
\lp
   \omega^{z}(u,v)^{2/p}+ \omega^{z'}(u,v)^{1/p} \omega^{x}(u,v)^{1/p} \rp  
   \omega_{1}^{x,\tilde{z},\tilde{z}'}(u,v)^{1/p}  
   \\
\times  \lp
  \omega^{x,\tilde{z},\tilde{z}'}(u,v)^{1/p} +\|\tilde{z}'\|_{\infty, [u,v]} +1
  \rp  \, .
\end{multline}
\end{lemma}

\begin{proof}
The proof of the lemma is an application of Lemma \ref{lem2.4}. 
Namely the existence of $R$ as a rough integral is ensured by general rough paths considerations (see e.g. \cite{G}). Then some elementary manipulations starting from the definition~\eqref{f2} of $R$ show that for $(u,s,v)\in \cs_{3}(\ll0,T\rr)$ we have 
\begin{eqnarray}\label{eqn.dRusv}
\delta R^{z\tilde{z}}_{usv} = ( \delta z_{us} - z_{u}' x^{1}_{us} ) \otimes\delta \tilde{z}_{sv} +  z_{us}' \int_{s}^{v} x^{1}_{sr}\otimes d\tilde{z}_{r}
\, ,
\end{eqnarray}
where we recall that $\delta  $ is defined as in \eqref{eqn.delta.def}. 

We first consider the case when $ (z,\tilde{z}) = (\tilde{z}, x)$, that is the  remainder $R^{\tilde{z}x}$  defined in \eqref{f2}. In this case  one can recast  \eqref{eqn.dRusv} as 
\begin{eqnarray}\label{eqn.drzx}
\delta R^{\tilde{z}x}_{usv} = ( \delta \tilde{z}_{us} - \tilde{z}_{u}' x^{1}_{us} ) \otimes x_{sv}^{1} +   \tilde{z}_{us}' \, x^{2}_{sv} . 
\end{eqnarray}
Applying the conditions in \eqref{eqn.sew.condition} we thus  get
\begin{eqnarray}\label{eqn.drz}
| \delta R^{\tilde{z}x}_{usv}  | \leq \omega^{\tilde{z}}(u,v)^{2/p} \cdot  \omega^{x}(u,v)^{1/p}    + \omega^{\tilde{z}'}(u,v)^{1/p} \cdot \omega^{x}(u,v)^{2/p}. 
\end{eqnarray}
Moreover, we have assumed that \eqref{f3} holds true for the increments $R^{\tilde{z}x}_{t_{k}t_{k+1}}$.
Hence a direct application of Lemma~\ref{lem2.4} implies that \eqref{eqn.drz} also holds when $\delta R^{\tilde{z}x}_{usv}  $ is replaced by $R^{\tilde{z}x}_{uv}$, that is relation \eqref{eqn.rbd} holds for $R^{\tilde{z}x}$ defined by \eqref{f2}. 

In order to prove \eqref{eqn.rbd} for a general $\tilde{z}$, 
let us first bound the integral $\int_{s}^{v} x^{1}_{sr}\otimes d\tilde{z}_{r}$ in \eqref{eqn.dRusv}. 
 To this aim, we observe that  a simple integration by parts (valid for rough integrals thanks to a limiting procedure on smooth approximations) yields the relation
 \begin{equation}\label{f4}
 \int_{s}^{v} x^{1}_{sr}\otimes d\tilde{z}_{r} = x^{1}_{sv}\otimes \delta \tilde{z}_{sv} -\int_{s}^{v} \delta\tilde{z}_{sr}\otimes d x_{r}
 \, .
 \end{equation}
 Next owing to our conditions \eqref{eqn.sew.condition} for $\tilde{z}$, the first term in the right hand side of \eqref{f4} is  bounded by
  \begin{eqnarray}\label{eqn.xtz}
|x^{1}_{sv}\otimes \delta \tilde{z}_{sv}|\leq \omega^{x}(s,v)^{1/p} \omega_{1}^{\tilde{z}}(s,v)^{1/p}. 
\end{eqnarray}
For the second term in the right-hand side of \eqref{f4}, let us write 
\begin{eqnarray*}
\int_{s}^{v}\delta \tilde{z}_{sr}\otimes dx_{r} = \tilde{z}_{s}' \otimes x^{2}_{sv} + R^{\tilde{z}x}_{sv}.  
\end{eqnarray*}
Since we have obtained that \eqref{eqn.rbd} holds for $R^{\tilde{z}x}$, for every $s\le u < v \le t$ we end up with 
\begin{multline}\label{eqn.xtz2}
\Big|\int_{u}^{v} \delta\tilde{z}_{ur} \otimes dx_{r} \Big| 
\leq |R^{\tilde{z}x}_{uv}|+ |\tilde{z}_{u}'| \omega^{x}(u,v)^{2/p}
\\
\leq \omega^{\tilde{z}}(u,v)^{2/p} \cdot  \omega^{x}(u,v)^{1/p}    + \omega^{\tilde{z}'}(u,v)^{1/p} \cdot \omega^{x}(u,v)^{2/p} + \|\tilde{z}'\|_{\infty, [u,v]} \, \omega^{x }(u,v)^{2/p}. 
\end{multline}
We can now safely plug   \eqref{eqn.xtz} and \eqref{eqn.xtz2} (with $u$ replaced by $s$) into relation \eqref{f4}. This yields  the    estimate
\begin{multline}\label{eqn.xtz.bd}
\Big|\int_{s}^{v} x^{1}_{sr}\otimes d\tilde{z}_{r} \Big| 
\leq \omega^{\tilde{z}}(u,v)^{2/p} \cdot  \omega^{x}(u,v)^{1/p}    + \omega^{\tilde{z}'}(u,v)^{1/p} \cdot \omega^{x}(u,v)^{2/p} 
\\
+  \|\tilde{z}'\|_{\infty, [u,v]}\cdot \omega^{x }(u,v)^{2/p}
+\omega^{x}(s,v)^{1/p} \omega_{1}^{\tilde{z}}(s,v)^{1/p}.  
\end{multline}
Let us now return to relation      \eqref{eqn.dRusv}. By a simple application of \eqref{eqn.sew.condition} one   discovers that    
\begin{eqnarray*}
| \delta R^{z\tilde{z}}_{usv}  | \leq \omega^{z}(u,v)^{2/p} \cdot  \omega_{1}^{\tilde{z}}(u,v)^{1/p}    + \omega^{z'}(u,v)^{1/p} \cdot\Big|\int_{s}^{v} x^{1}_{sr}\otimes d\tilde{z}_{r} \Big| . 
\end{eqnarray*}
Inserting  \eqref{eqn.xtz.bd} into this relation, we thus obtain 
\begin{multline}
\Big|
\delta R^{z\tilde{z}}_{usv}  
 \Big|
   \leq  K_{p}\lp
   \omega^{z}(u,v)^{2/p}+ \omega^{z'}(u,v)^{1/p} \omega^{x}(u,v)^{1/p} \rp  
   \omega_{1}^{x,\tilde{z},\tilde{z}'}(u,v)^{1/p} 
    \\
\times  \lp
  \omega^{x,\tilde{z},\tilde{z}'}(u,v)^{1/p} +\|\tilde{z}'\|_{\infty, [u,v]} +1
  \rp .
  \label{eqn.rzzt}
\end{multline}
Note that the right-hand side of \eqref{eqn.rzzt} is equal to 
$K_{p}\omega^{R}(u,v)$   defined in  \eqref{eqn.wr}. 
Taking \eqref{f3} into account, 
another use of Lemma \ref{lem2.4}
(together with an application of  \cite[Exercise 1.9 (iii)]{FV} to show that $\om^{R}$ is a control),
     proves  our claim 
\eqref{eqn.rbd}. 
\end{proof}

\subsection{Interpolation of the Euler method}\label{section4.3}
  In this section we shall extend our Euler scheme to a process in continuous time and obtain some uniform bounds. Specifically, recall that the Euler approximation $y^{n}$ is defined on $\ll0,T\rr$ by \eqref{e4} and for  convenience we will take $V_{0}\equiv 0$.  For  $t$ in the continuous interval $[0,T]$ we shall use the following interpolation:
\begin{align}\label{eqn.euler3}
\delta y^{n}_{t_{k}t} 
=& V(y^{n}_{t_{k}}) \delta x_{t_{k}t } + \frac12 \sum_{j=1}^{d} \partial V_{j} V_{j} (y^{n}_{t_{k}})   (t-t_{k})^{2H}
 ,
 \qquad t\in [t_{k}, t_{k+1}] . 
\end{align}
{Recall that  $\bfx=S_{2}(x)=(x^{1}, x^{2})$ denotes the rough path above $x$ (as given in Definition \ref{def:rough-path}).}

  In the sequel we will also need some continuous interpolations of the processes $q$ and $q^{b}$, which had been defined on the grid in \eqref{eqn.q}. Namely  
for $(s,t)\in \cs_{2}([  0,T ])$ such that $t_{k_{1}}\leq s< t_{k_{1}+1}<\cdots<t_{k_{2}-1}\leq t <t_{k_{2}}$, 
we define
\begin{align}
 q_{st}^{ij}=
   \sum_{k_{1}\leq k<k_{2}} \lp   x_{t_{k}\vee s, t_{k+1}\wedge t}^{2,ij} -\frac12 (
     t_{k+1}\wedge t
    -
    t_{k}\vee s)^{2H} \mathbf{1}_{\{i=j\}} \rp
    \label{eqn.qc}
\\
q^{b, ij}_{st} =   \sum_{k_{1}\leq k<k_{2}} \lp   b_{t_{k}\vee s, t_{k+1}\wedge t}^{2,ij} -\frac12 (
     t_{k+1}\wedge t
    -
    t_{k}\vee s)^{2H} \mathbf{1}_{\{i=j\}} \rp . 
    \label{eqn.qbc} 
\end{align}
With the above definition \eqref{eqn.qc} in hand, we will also extend the definition of $\omega$ from the grid to $[0,T]$. 

\begin{lemma}
Recall that $\omega$ has been  defined on $\cs_{2}(\ll0,T\rr)$ by \eqref{eqn.control} 
for a fixed partition length parameter $n$, and that we are considering $p<3$. 
We now extend $\omega$ to $\cs([0,T])$ by setting:
 \begin{align}\label{eqn.omegac}
 \omega (s,t) = \| \bfw\|_{p\text{-var}; [ s,t] }^{p} + \|q\|_{p/2\text{-var};  [ s,t] }^{p/2} + \|q^{b}\|_{p/2\text{-var};  [ s,t] }^{p/2}  +|t-s|.
\end{align}
Then
$\omega$   is a control on $[0,T]$. In other words, $\omega$ is super-additive, continuous and vanishes  on the diagonal. 
\end{lemma}
\begin{proof}
We first note that  
\begin{align}\label{eqn.qcont}
 \|q\|_{p/2\text{-var};  [ s,t] }^{p/2} \leq   \|q\|_{p/2\text{-var};  \ll s,t\rr }^{p/2}   +\| \bfw\|_{p\text{-var}; [ s,t] }^{p} + |t-s| ,
\end{align}
and the same kind of inequality holds true for $q^{b}$.
 Therefore $\omega(s,t)$ is finite almost surely.  
Thanks to the definition \eqref{eqn.omegac} of $\omega$ it is also readily checked that the superadditivity and zero on the diagonal properties hold true.

It remains to show the continuity of $\omega$.  To this aim, taking into account the definition~\eqref{eqn.omegac} of $\omega$, it is easily seen that we only have to focus on the increments $q$ and $q^{b}$ in \eqref{eqn.qc}-\eqref{eqn.qbc}. Moreover $q$ and $q^{b}$ are handled exactly in the same way. Hence we will just focus our attention on $\tilde{\omega}$ given by  \begin{eqnarray}\label{eqn.omegat}
\tilde{\omega}(s,t):= \|q\|_{p/2\text{-var};[s,t]}^{p/2} \, .
\end{eqnarray}
 Take  $s<u<t<\eta(u)+\Delta$, where we recall from Notation \ref{general-notation} that $\eta(u)$ is the largest $t_{k}\in \ll0,T\rr$ such that $t_{k}\leq u$. 
In the following we  show that $\tilde{\omega} (s,t)  -   \tilde{\omega} (s,u)\to0$ as $t-u\to0$, which is one of the main steps towards the continuity of $\tilde{\omega} $.
 
 Owing to the definition \eqref{eqn.omegat} of $\tilde{\omega}$, for any $\ep>0$, we can find a partition of $[s,t]$, denoted by $s=v_{0}<\cdots<v_{N}=t$ such that 
  \begin{eqnarray}\label{eqn.tomega}
\tilde{\omega}(s,t)\leq \sum_{i=0}^{N-1}|q_{v_{i }v_{i+1}}|^{p/2}+\ep \, .
\end{eqnarray}
 Suppose that $u\in [v_{i_{0}}, v_{i_{0}+1}]:=[v,v']$.  
  Then we can bound the summation in  \eqref{eqn.tomega} by the following 
 \begin{eqnarray}\label{eqn.wtbd}
\tilde{\omega} (s,t) 
&\leq&
\sum_{i=0}^{i_{0}-1}|q_{v_{i }v_{i+1}}|^{p/2}+|q_{vv'}|^{p/2}+\sum_{i=i_{0}+1}^{N-1}|q_{v_{i }v_{i+1}}|^{p/2}
+\ep \notag\\
&\leq&   \tilde{\omega} (s,v)+|q_{vv'}|^{p/2}+\tilde{\omega} (v',t)+\ep    . 
\end{eqnarray}
Now go back to \eqref{eqn.qcont} for $v'$ and $t$ and pick $u$ close enough to $t$ so that $v'\in[u,t]$ satisfies $ \|q\|_{p/2\text{-var}; \ll v', t \rr} =0$ (since our grid $\ll u',t \rr$ has fixed mesh $T/n$, this is easily seen when $u\to t$, owing to our expression \eqref{eqn.qc}). One can thus recast \eqref{eqn.qcont} as 
\begin{eqnarray}\label{eqn.wtbd2}
\tilde{\omega}(v', t) := \|q\|_{p/2\text{-var}; [v', t]}^{p/2} \leq \|\bfw\|_{p \text{-var}; [v', t]}^{p} +|t-v'|. 
\end{eqnarray}
It is then easily seen from \eqref{eqn.wtbd2} that $\lim_{v'\to t}\tilde{\omega}(v', t) = 0$. Since $u<v'<t$, we will pick $u$ close enough to $t$ so that $\tilde{\omega}(v', t)\leq \ep$. Plugging this information into \eqref{eqn.wtbd}, we obtain 
\begin{eqnarray}\label{eqn.wtst}
\tilde{\omega}(s,t) \leq \tilde{\omega}(s,v) + |q_{vv'}|^{p/2} +2\ep . 
\end{eqnarray}
In addition, if $u\to t$ we also have $|v'-u|\to 0$. Therefore, basic continuity properties of $q$ ensure  that $ \big| |q_{vv'}|^{p/2} - |q_{vu}|^{p/2} \big|\leq \ep $ if $u$ is close enough to $t$. Hence \eqref{eqn.wtst} becomes 
\begin{equation*}
\tilde{\omega}(s,t) \leq 
\tilde{\omega}(s,v) +|q_{vu}|^{p/2}+3\ep 
\leq 
\tilde{\omega}(s,u)  +3\ep , 
\end{equation*}
where we have used the super-additivity property of $\tilde{\omega}$ for the second inequality. Since $\tilde{\omega}(s,t)\geq \tilde{\omega}(s,u)$ by monotonicity properties, we have obtained 
\begin{eqnarray}\label{eqn.wtuc}
| \tilde{\omega}(s,t) - \tilde{\omega}(s,u)    |\leq 3\ep, 
\end{eqnarray}
 for all $(s,u,t) \in \cs_{3} ([0,T])$ such that $|t-u|$ is sufficiently small. Since $\ep$ in \eqref{eqn.wtuc} can be  arbitrarily small, this proves that 
 $\lim_{u\to t} \tilde{\omega}(s,u) = \tilde{\omega}(s,t) $. The same kind of arguments also show that 
 $\lim_{u\to s} \tilde{\omega}(u,t) = \tilde{\omega}(s,t) $, which completes our proof.
  \end{proof}

We now go back to the interpolated version of our Euler scheme $y^{n}$. In the following we show that $(y^{n}, x,b)$ is a rough path, which is an important step in the convergence analysis. 
\begin{lemma}\label{lem5.7}
Consider the interpolated Euler scheme introduced in \eqref{eqn.euler3}. Recall that $x$ is our driving fBm and $b$ is another fBm with parameter $H>1/3$, independent of $x$. Also recall that the augmented process $\emph{w} = (x,b)$ has been introduced in \eqref{eqn.w.def}. We assume that the vector field $V$ sits in $C^{4}_{b}$. Then the following holds true. 

  \noindent\emph{(i)} 
  Denote by $Z$ the couple   $Z= (y^{n},   \emph{w})$.    Then $Z$ admits a lift $S_{2}(Z)$ according to Definition~\ref{def:rough-path}. Moreover, recalling the sets $S_{0}$, $S_{1}$ in \eqref{eqn.s01}, consider  $s_{j}\in S_{0}\cup S_{1}$. Then for all   $(s,t) \subset \cs_{2}([s_{j}, s_{j+1}])  $ we have the following uniform bound in $n$:
\begin{align}\label{eqn.s2z}
 \|S_{2}(Z)\|_{p\tvr, [  s,t]} \leq K \cdot \omega(s,t)^{1/p}  , 
\end{align}
where $p>1/H$, $K$ is a constant depending on $V$ and where the control $\omega$ is defined in \eqref{eqn.omegac}.

  \noindent\emph{(ii)} 
  Under the setting of Theorem \ref{thm.xirr}, we now consider the vector-valued stochastic process
    $\tilde{Z}= (y^{n}, D_{r}y^{n}, D_{r'}D_{r}y^{n},  \emph{w})$ for $r,r'\in [0,T]$.   
    We still assume that  $(s,t) \subset (s_{j}, s_{j+1})  $   with $s_{j}\in S_{0}\cup S_{1}$. Then we have (still uniformly in $n$):
\begin{align}\label{eqn.s2zt}
 \|S_{2}(\tilde{Z})\|_{p\tvr, [  s,t]} \leq K \cdot \omega(s,t)^{1/p} \cdot \cg^{2},  
\end{align}
where the quantities $\cg$ have been introduced in   \eqref{e.cg}.
\end{lemma}

\begin{proof}
We first prove    (i).  
Take $(s,t) \subset (s_{j}, s_{j+1})  $   such that $s_{j}\in S_{0}\cup S_{1}$. Theorem \ref{thm.xirr}, applied for $L=0$, shows that for $(s,t)\in \cs_{2}(\ll s_{j},s_{j+1} \rr)$ we have
\begin{eqnarray}\label{eqn.yn.bd}
| \delta y_{st}^{n} | \leq K\omega(s,t)^{1/p}|S_{0}\cup S_{1} \cup S_{2}|
\quad
\text{and}
\quad
| \delta y_{st}^{n} - V(y_{s}^{n}) \delta x_{st} | \leq K\omega(s,t)^{2/p} \, ,
\end{eqnarray}
where $\omega$ is the control given in \eqref{eqn.omegac}. Using standard interpolation methods, it can be shown    in a straightforward way  that the  relation for $|\delta y^{n}_{st}|$ in \eqref{eqn.yn.bd} still   holds if  $(s,t) \in \cs_{2} ([s_{j}, s_{j+1}])$. So in order to prove \eqref{eqn.s2z} 
it remains  to show that $\int_{s}^{t} y^{n}_{su} \otimes d\text{w}_{u} $ and $\int_{s}^{t} y^{n}_{su} \otimes dy^{n}_{u} $ are bounded by $\omega(s,t)^{2/p}$.

 Consider the following remainder process for $(s,t) \subset (s_{j}, s_{j+1})  $   such that $s_{j}\in S_{0}\cup S_{1}$: 
\begin{align}\label{eqn.Rsty}
R_{st}=   
\int_{s}^{t} \lp
\delta
y^{n}_{su}
-   V(y^{n}_{s})    \delta x_{su} 
    \rp
     \otimes  d \text{w}_{u}.  
\end{align}
Note that $R$ is a remainder of the form $R^{y^{n}w}$, defined as in \eqref{f2}. In order to apply Lemma~\ref{lem.sew.discrete} to this remainder, we need to check that $R_{t_{k}t_{k+1}}\leq \omega (t_{k}, t_{k+1})^{3/p}$ as in \eqref{f3}. Now according to \eqref{eqn.Rsty} we have 
\begin{eqnarray}\label{eqn.rkk}
R_{t_{k}t_{k+1}}=   
\int_{t_{k}}^{t_{k+1}} \lp
\delta
y^{n}_{t_{k}u}
-   V(y^{n}_{t_{k}})    \delta x_{t_{k}u} 
    \rp
     \otimes  d \text{w}_{u}.  
\end{eqnarray}
Furthermore, owing to our interpolation formula \eqref{eqn.euler3}, for all $u\in [t_{k}, t_{k+1}]$ we have 
\begin{eqnarray*}
 \delta y^{n}_{t_{k}u} 
-
 V(y^{n}_{t_{k}}) \delta x_{t_{k}u}  
 &=& \frac12 \sum_{j=1}^{d} \partial V_{j} V_{j} (y^{n}_{t_{k}})   (u-t_{k})^{2H}. 
\end{eqnarray*}
Reporting this identity into \eqref{eqn.rkk}, we end up with 
\begin{eqnarray}\label{eqn.rkkn}
R_{t_{k}t_{k+1}} = \frac12 \sum_{j=1}^{d} \partial V_{j}V_{j} (y^{n}_{t_{k}}) \otimes \int_{t_{k}}^{t_{k+1}} (u-t_{k})^{2H} d\text{w}_{u} . 
\end{eqnarray}
The stochastic integral in the right-hand side of \eqref{eqn.rkkn} can be interpreted in the Young sense, and it is easy to see that $|R_{t_{k}t_{k+1}}|\lesssim \omega (t_{k}, t_{k+1})^{3/p}$, 
where $\omega$ is still the control introduced in~\eqref{eqn.omegac}. This proves \eqref{f3} for the remainder $R$, and therefore one can safely apply Lemma~\ref{lem.sew.discrete} in order to get that, provided $s_{j}\in S_{0}\cup S_{1}$,  
  \begin{align}\label{eqn.rstyw}
|R_{st}| \leq K \omega (s,t)^{3/p}   
\qquad \text{ for }
(s,t) \in \cs_{2}(\ll s_{j}, s_{j+1}\rr) . 
\end{align}
Going back to \eqref{eqn.Rsty}, observe that we have 
\begin{eqnarray*}
\int_{s}^{t} \delta y^{n}_{su} \otimes d\text{w}_{u} = V(y^{n}_{s}) \int_{s}^{t} \delta x_{su} \otimes d\text{w}_{u} +R_{st}.  
\end{eqnarray*}
Invoking \eqref{eqn.rstyw} and since $x$ is part of the rough path $\text{w}$, we easily get 
\begin{align}\label{eqn.ydw}
\Big|\int_{s}^{t} \delta y^{n}_{su}\otimes  d\text{w}_{u}  \Big| \leq K \omega (s,t)^{2/p} , 
\qquad \text{ for }
(s,t) \in \cs_{2}(\ll s_{j}, s_{j+1}\rr) . 
\end{align}
In the following we extend the above estimate for $\int_{s}^{t} \delta y^{n}_{su} d\text{w}_{u}$ to any $(s,t) \in \cs_{2}([ s_{j}, s_{j+1}]) $. That is we take $(s,t)$ such that $s\in [t_{k}, t_{k+1}]$ and $t\in [t_{k'}, t_{k'+1}]$.  We   write
\begin{eqnarray*}
\int_{s}^{t}\delta y^{n}_{su}\otimes d\text{w}_{u} &=&
\int_{s}^{t_{k+1}}\delta y^{n}_{su} \otimes d\text{w}_{u}
+
\int_{t_{k+1}}^{t_{k'}} \delta y^{n}_{su}  \otimes d\text{w}_{u}
+
\int_{t_{k'}}^{t} \delta y^{n}_{su} \otimes d\text{w}_{u}
\\
&=&
\int_{s}^{t_{k+1}}\delta  y^{n}_{su} \otimes d\text{w}_{u}
+
\delta y^{n}_{st_{k+1}}
\otimes   \delta\text{w}_{t_{k+1}t_{k'}}
+
\int_{t_{k+1}}^{t_{k'}} \delta y^{n}_{t_{k+1}u} \otimes d\text{w}_{u}
\\
&&\quad+
\delta y^{n}_{st_{k'}} \otimes   \delta\text{w}_{t_{k'} t}
+
\int_{t_{k'}}^{t} \delta y^{n}_{t_{k'}u} \otimes d\text{w}_{u}
=: \sum_{i=1}^{5}I_{i}
.  
\end{eqnarray*}
Let us bound the terms $I_{1},\ldots,I_{5}$ above. First it follows from \eqref{eqn.ydw} that $|I_{3}|\leq K \omega (s,t)^{2/p}$. It is also clear that  $|I_{2}|$ and $|I_{4}|$ are bounded by the same estimate $\omega (s,t)^{2/p}$. 
In order to bound $I_{1}$, observe that according to our interpolation formula \eqref{eqn.euler3} we have 
\begin{multline}\label{eqn.i1}
I_{1} 
= V(y^{n}_{t_{k}}) \int_{s}^{t_{k+1}} \delta x_{su}\otimes d\text{w}_{u} \\
+
\frac12 \sum_{j=1}^{d} \partial V_{j}V_{j} (y^{n}_{t_{k}}) \otimes \int_{s}^{t_{k+1}} \lc (u-t_{k})^{2H}-(s-t_{k})^{2H}\rc d\text{w}_{u}. 
\end{multline}
It is clear that the first integral in \eqref{eqn.i1} is bounded by $\omega(s,t_{k+1})^{2/p}$. Note also that $(u-t_{k})^{2H}-(s-t_{k})^{2H} \leq (u-s)^{2H}$. So applying Young's inequality we obtain that the second integral in \eqref{eqn.i1} is bounded by $(t_{k+1}-s)^{2H}\omega(s,t_{k+1})^{1/p}$. Combining these two estimates  we obtain that $|I_{1}|\lesssim \omega (s,t_{k+1})^{2/p}$. The term $I_{5}$ is bounded   in the similar way. Putting together our upper bounds on $I_{1}$, \dots, $I_{5}$, we have thus obtained that \eqref{eqn.ydw} holds for any $(s,t)\in \cs_{2}([s_{j}, s_{j+1}])$. 
Summarizing our considerations so far, we have proved that 
\begin{eqnarray}\label{eqn.m1}
\| M^{1}\|_{p\text{-var}; [s,t]} \leq K \omega (s,t) \qquad
\text{with } M_{uv}^{1}:= \int_{u}^{v} \delta y^{n}_{ur}\otimes d\text{w}_{r}. 
\end{eqnarray}
The proof is now finished along the same arguments. Namely the increment $M_{uv}^{2} := \int_{u}^{v} \delta y^{n}_{ur} \otimes dy^{n}_{ur}$ can be bounded similarly to $M^{1}$, leading to the same inequality as \eqref{eqn.m1}. This proves our claim \eqref{eqn.s2z}. 

Eventually \eqref{eqn.s2zt} is also obtained along the same lines. We just apply  Theorem \ref{thm.xirr} with $L=1,2$ 
 in the arguments, and also write $\cg^{2}\cdot \omega (s,t)$ instead of $\omega (s,t)$ in all inequalities for the remainders. This completes our proof.   
\end{proof}

\subsection{Integrability of some linear equations}\label{section.linear}
Our convergence estimates are based on linearization procedures. In this section we bound some related linear differential equations. We start by defining the objects we wish to study. 
\begin{Def}\label{def.v1}
Recall that every $V^{i}$ has to be seen, for $i=1,\dots, m$, as a smooth vector field on $\RR^{m}$. Let $y^{n}$ be the interpolated scheme \eqref{eqn.euler3}. Then for $i=1,\dots, m$ we define an averaged $\RR^{d\times m}$-valued process $\tilde{V}^{i}(t) = \{\tilde{V}_{ ji'}^{i}(t) ; \, j =1,\dots, d, \, i'=1,\dots, m\}$. This process is indexed by $t\in [0,T]$ and is given by 
\begin{eqnarray*}
\tilde{V}_{  ji'}^{i}(t) = 
\int_{0}^{1} \partial_{i'} V^{i}_{j}(\theta y_{t} +(1-\theta)y^{n}_{t}) d\theta  . 
\end{eqnarray*} 
We also define a $(\RR^{d\times m})^{m} $-valued process    as $\tilde{V}(t) = (\tilde{V}^{1} (t), \dots, \tilde{V}^{m}(t))$. 
\end{Def}

We are now ready to define the linear equation we wish to analyze. 
\begin{Def}
Let $\tilde{V} $ be the $(\RR^{d\times m})^{m} $-valued process introduced in Definition \ref{def.v1}. We will call $\Ga$ the $\RR^{m\times m}$-valued solution to the following systems of equations on $[0,T]$: 
\begin{eqnarray}\label{eqn.Ga}
\Ga_{t}^{ii'} = \id_{ii'} + \sum_{j=1}^{d}\sum_{i''=1}^{m} \int_{0}^{t} \tilde{V}^{i}_{ji''}(s) \Ga_{s}^{i''i'}dx_{s}^{j},
\quad\text{for}\quad i,i'\in\{1,\ldots,m\}. 
\end{eqnarray}
For conciseness   we will simply write \eqref{eqn.Ga} as  $\Ga_{t} = \id +\int_{0}^{t} \tilde{V} (s) \Ga_{s} \, dx_{s} $. 
We also denote by $\La$ the inverse of $\Ga$, namely $\La$ is defined by the relation $\La_{t} \Ga_{t} \equiv\id$.
\end{Def}

Our next lemma presents an important estimate for the processes $\Ga$ and $\La$ defined above.
 
 \begin{lemma}\label{lem.ig}
 Let the assumption  be as in Theorem \ref{thm.xirr}.  
 Let $p>\frac{1}{H}$ and $q$ be such that $\frac{1}{p}+\frac{1}{q}>1$. 
Let $\Ga$ and $\La$ be defined in \eqref{eqn.Ga}. Then

\noindent\emph{(a)}
For all $(s,t)\in\cs_{2}([0,T])$ we have
\begin{multline}\label{eqn.gala}
 |\delta \Ga_{st}| + \sup_{r\in[0,T]}|  D_{r}(\delta\Ga_{st})| + \sup_{r,r'\in[0,T]} | D_{rr'}^{2} (\delta\Ga_{st})|  
 \\
  \leq 
 K  \cdot \omega(s,t)^{1/p} \cdot |S_{0}\cup S_{1}\cup S_{2}| \cdot \cm_{0}\cdot \cm_{1}\cdot \exp (\cn)
, 
\end{multline}
where $\omega$ is defined in \eqref{eqn.omegac} and $\cn$ is some  random variable such that $\cn^{1/q}$ has  a Gaussian tail. The relation still holds when $\Ga$ is replaced by $\La$. 

\noindent\emph{(b)}  Both  processes $\Ga$ and $\La$ and their Malliavin derivatives  are uniformly integrable. Precisely, for all $p\geq1$ we have we have
\begin{eqnarray}\label{e.int}
\mE\lc
  \sup_{n\in \NN, r,r',t\in[0,T]}\lp|  \Ga_{t}|^{p} +  |  D_{r}( \Ga_{t})|^{p} + | D_{rr'}^{2} ( \Ga_{t})|^{p} \rp \rc <\infty . 
\end{eqnarray}

\end{lemma}
\begin{proof}
Applying 
Corollary \ref{cor.m1} and Theorem \ref{thm.m0} to the right-hand side of \eqref{eqn.gala} we conclude the integrability relation in \eqref{e.int}. It thus  remains to prove relation \eqref{eqn.gala}. 

In the following we   prove the estimate \eqref{eqn.gala} for $\Ga$.  Note that due to the fact that the initial condition $\id$ in~\eqref{eqn.Ga} is nondegenerate, the process  $\La$ is well-defined and satisfies a differential equation which is very similar to $\Ga$ (see e.g. \cite{CHLT, HLN1} for more details). Therefore the estimate of $\La$ in~\eqref{eqn.gala} can be obtained by following the same steps as for $\Ga$.   The proof for $\La$ and its derivatives is thus omitted.

With our Definition \ref{def.v1} in mind, 
let us first introduce an auxiliary process $\xi$ given, for $i,i'' \in \{ 1, \dots, m \}$ and $t\in [0,T]$, by 
\begin{eqnarray*}
\xi_{t}^{ii''} = \sum_{j=1}^{d}  \int_{0}^{t} \tilde{V}^{i}_{ji''}(s)  dx_{s}^{j}.  
\end{eqnarray*}
We now separate the estimates for $\xi$ into two different cases
\begin{enumerate}[wide, labelwidth=!, labelindent=0pt, label=(\roman*)]
\setlength\itemsep{.1in}

\item
\emph{Case $(s,t) \subset (s_{j}, s_{j+1})  $   such that $s_{j}\in S_{0}\cup S_{1}$.}
In this situation, since $(y^{n}, x, b)$ is a rough path (see  Lemma \ref{lem5.7} (i)) and $y$ is a process controlled by $x$, it is readily checked that  for   all $(s,t) \subset (s_{j}, s_{j+1})  $   such that $s_{j}\in S_{0}\cup S_{1}$ we have:
\begin{align}\label{eqn.s2xi}
 \|S_{2}(\xi)\|_{p\tvr, [  s,t]} \leq K  \cdot \omega(s,t)^{1/p}.    
\end{align} 
Furthermore, note that one can recast equation \eqref{eqn.Ga} as a linear system of the form 
\begin{eqnarray}\label{eqn.ga2}
d\Ga_{t}^{ii'} =\sum_{i''=1}^{m} \Ga_{t}^{i''i'} d\xi_{t}^{ii''}. 
\end{eqnarray}
Observe that the path $\xi$ is a functional of the process $Z$ introduced in Lemma \ref{lem5.7}. Now 
we recall from \cite[Theorem 10.53]{FV} that for a linear equation like \eqref{eqn.ga2}, there exist two constants $C_{1}$, $C_{2}$ such that 
\begin{eqnarray}\label{eqn.gz}
|S_{2}(\Ga, Z)_{st}| 
&\leq& 
C_{1} |\Ga_{s}| \cdot \|S_{2}(\xi)\|_{p\text{-var}, [s,t]} \cdot \exp (C_{2} \|S_{2}(\xi)\|_{p\text{-var}, [s,t]})
\nonumber
\\
&\leq& 
C_{1} |\Ga_{s}| \cdot \omega (s, t)^{1/p} \exp ( C_{2}  \omega (s, t)^{1/p} ) , 
\end{eqnarray}
where the second relation stems from \eqref{eqn.s2xi}. In addition, we have chosen $s_{j}\in S_{0}\cup S_{1}$. Therefore, one can simplify \eqref{eqn.gz} and obtain  that for any $(s,t)\in \cs_{2} ([s_{j},s_{j+1}])$,
\begin{eqnarray}\label{eqn.gai2}
\|S_{2}(\Ga, Z)\|_{p\text{-var}, [s,t]} \leq K \omega (s,t)^{1/p}  \cdot |\Ga_{s}| \, .
\end{eqnarray}

\item
\emph{Case $(s,t) \subset (s_{j}, s_{j+1})  $   such that $s_{j}\in S_{2}$.}
For $s_{j}\in S_{2}$   equation \eqref{eqn.Ga} is a linear equation driven by $x$ and so we can apply the integrability result in \cite{CLL} to get
\begin{eqnarray}\label{eqn.gas2}
\|S_{2}(\Ga)\|_{p\text{-var}, [s,t]} \leq K |\Ga_{s}| \cdot\|\bfx\|_{p\text{-var}, [s,t]}  \cdot \exp (\cn_{s_{j}s_{j+1}}) , 
\end{eqnarray}
where     $\cn_{s_{j}s_{j+1}}$ is a random variable such that, denoting  $\cn:=\sum_{s_{j}\in S_{2}} \cn_{s_{j}s_{j+1}}$, the random variable  $\exp(K\cdot\cn)$ is integrable for any constant $K>0$. 
\end{enumerate}

Our estimate \eqref{eqn.gala} for $\delta \Ga_{st}$ is now easily obtained. Namely we iterate 
  \eqref{eqn.gai2}-\eqref{eqn.gas2}.  
  This yields 
\begin{eqnarray*}
|\delta \Ga_{st}| \leq K \lp \sum_{s_{j}\in S_{0}\cup S_{1}} \omega (s_{j}, s_{j+1})^{1/p} + \sum_{s_{j}\in S_{2}} \|\bfx\|_{p\text{-var}; [s_{j}, s_{j+1}]} \rp \cdot \cm_{0}\cdot \cm_{1}\cdot    \exp (\cn) . 
\end{eqnarray*}
 Finally taking into account the definition of $S_{i}$, $i=0,1,2$ it follows that 
 \begin{eqnarray*}
|\delta \Ga_{st}| \leq K \omega (s, t)^{1/p} \cdot |S_{0}\cup S_{1}\cup S_{2}| \cdot \cm_{0}\cdot \cm_{1}\cdot \exp (\cn). 
\end{eqnarray*}
Namely we have proved \eqref{eqn.gala} for $\Ga$. It remains to upper bound the Malliavin derivatives of~$\Ga$. 

Recall that $\Ga$ satisfies equation \eqref{eqn.Ga}, with $\tilde{V}$ given in Definition \ref{def.v1}. For sake of clarity, the remainder of our computations will be done assuming that all our quantities are real-valued (we will therefore drop the indices from our next equations). Moreover according to our standing assumptions, the process $\tilde{V}$ is Malliavin differentiable. Hence using standard arguments for the differentiation of rough differential equations (see \cite{CF, CHLT, Inahama, NS}) we get that $D_{r}\Ga_{t}$ satisfies the linear equation:  
\begin{eqnarray}\label{eqn.Dtau}
D_{r}\Ga_{t} &=& D_{r} \int_{0}^{t} \tilde{V}(s) \Ga_{s}dx_{s}
\nonumber
\\
&=&     \tilde{V}(r) \Ga_{r} +    \int_{r}^{t} \partial \tilde{V}(s) (\theta D_{r}y_{s}+(1-\theta)D_{r}y^{n}_{s}) \Ga_{s}dx_{s} +\int_{r}^{t} \tilde{V}(s) D_{r}\Ga_{s}dx_{s}. 
\end{eqnarray}
Therefore one can use the variation of constant method, similarly to \cite[equation (2.7)]{Inahama}, in order to get the following representation for $D_{r}\Ga_{t}$:  
\begin{eqnarray}\label{eqn.dga}
D_{r}\Ga_{t} =\Ga_{t}^{r} \tilde{V}(r) \Ga_{r} +
\Ga_{t}^{r} \int_{r}^{t} \La_{s}^{r}  \partial \tilde{V}(s) (\theta D_{r}y_{s}+(1-\theta)D_{r}y^{n}_{s}) \Ga_{s}dx_{s} , 
\end{eqnarray}
where $\{\Ga^{r}_{t}; \, t\in[ r,T]\}$ is the solution of equation \eqref{eqn.ga2} such that $\Ga_{r}^{r}=\id$ and $\La_{t}^{r}$ is the inverse of $\Ga^{r}_{t}$. Note that because $\Ga$ and $\Ga^{r}$ satisfy the same equation with different initials,  the  estimate of $\Ga$ in \eqref{eqn.gas2} also holds for $\Ga^{r}$.  
In order to estimate $D_{r}\Ga_{t}$, it thus remains to get the estimate \eqref{eqn.gala} for  the integral 
\begin{eqnarray}\label{eqn.dgbd}
\int_{r}^{t} \La_{s}^{r}  \partial \tilde{V}(s) (\theta D_{r}y_{s}+(1-\theta)D_{r}y^{n}_{s}) \Ga_{s}dx_{s}
\end{eqnarray}
 in \eqref{eqn.dga}. Recall that  
$\Ga$ is the solution of the linear system   \eqref{eqn.ga2} driven by $Z$, where we recall that $Z= (y^{n}, \text{w})$. 
According to \eqref{eqn.s2zt}   $(D_{r}y, D_{r}y^{n}, Z)$ is a rough path. 
So $\Ga$ can  also be considered as   the solution of a linear system   driven by $(D_{r}y, D_{r}y^{n}, Z)$. 
Hence along the same line as for \eqref{eqn.gz} we can estimate the quantity \eqref{eqn.dgbd}, and thus we obtain the bound~\eqref{eqn.gala}  for  $D_{r}\Ga_{t} $.  

We turn to  the equation satisfied by $D^{2}_{rr'}\Ga$. Differentiating \eqref{eqn.Dtau}, we let the patient reader check that the second derivative verifies a linear equation of the form 
\begin{eqnarray}\label{eqn.ddga}
D_{rr'}^{2} \Ga_{t} = D_{r'}[  \tilde{V}(r) \Ga_{r}]  +      D_{r }[  \tilde{V}(r') \Ga_{r'}] 
 + \ce_{rr'}(t)+ \int_{r\vee r'}^{t} \tilde{V}(s) D^{2}_{rr'}\Ga_{s}dx_{s} ,
\end{eqnarray}
where the term $\ce_{rr'}(t)$ is defined by 
\begin{eqnarray*}
\ce_{rr'}(t)
&=&
\int_{r\vee r'}^{t} \partial \tilde{V}(s) (\theta D_{r}y_{s}+(1-\theta)D_{r}y^{n}_{s}) D_{r'}\Ga_{s}dx_{s}
\\
&&
+  \int_{r\vee r'}^{t} \partial \tilde{V}(s) (\theta D_{r'}y_{s}+(1-\theta)D_{r'}y^{n}_{s}) D_{r}\Ga_{s}dx_{s}
\\
&&
+  \int_{r\vee r'}^{t} \partial^{2} \tilde{V}(s) (\theta D_{r'}y_{s}+(1-\theta)D_{r'}y^{n}_{s}) (\theta D_{r}y_{s}+(1-\theta)D_{r}y^{n}_{s})   \Ga_{s}dx_{s}.
\end{eqnarray*}
It is clear   that the process   $  D_{rr' }^{2}\Ga  $ satisfies   a linear equation system  analogous to 
\eqref{eqn.Dtau}. The   estimate  can  thus  be obtained by following the same   arguments as above, invoking again~\cite{Inahama}. This completes the proof of \eqref{eqn.gala}.   
\end{proof}

\subsection{A decomposition  of the error process}\label{section.decompose}

In \cite[equations (6.14) and (7.6)]{LT}, we have decomposed the error process $y_{t}-y^{n}_{t}$ according to the Jacobian of the equation and some remainder terms. In the following proposition we get a similar decomposition, adapted to our needs for the weak convergence estimates. Notice that similarly to what we did in Section~\ref{section.linear}, we will drop the indices from our formulae below for sake of readability.

\begin{lemma}\label{prop.decomposition}
We work under the conditions of Lemma \ref{lem5.7}. 
Let $y$ and $y^{n}$ be the solutions of  equation \eqref{e2.1} and the Euler scheme \eqref{eqn.euler3}, respectively. 
Let $\Ga $ and $\La$ be respectively the solution of   equation
\eqref{eqn.Ga} and its inverse $\La = \Ga^{-1}$. 
We set $\eta(s) = t_{k}$ for $s\in [t_{k}, t_{k+1})$. 
For $t\in [0,T]$ we also define  
\begin{eqnarray}\label{eqn.I}
I_{t} &=&
\frac12
\int_{0}^{t} 
\partial V    \partial VV (y^{n}_{\eta(s)}) \cdot  (s-\eta(s))^{2H}
 d x_{s} 
 +
\int_{0}^{t} 
\lp
\int_{\eta(s)}^{s} \int_{\eta(s)}^{u}\partial^{2} V (y^{n}_{v}) dy^{n}_{v} dy^{n}_{u}
\rp
 d x_{s} 
 \nonumber
 \\
 &=:& I_{t}^{1} +I_{t}^{2} . 
\end{eqnarray}
Then the difference $y_{t}^{n}-y_{t}$ can be decomposed as  
\begin{eqnarray}\label{eqn.ymyj}
y_{t}-y^{n}_{t}&=& \sum_{e=1}^{5} J_{t}^{e},
\end{eqnarray}
where the processes $ J^{1}_{t}, J^{2}_{t}, J^{3}_{t}$ are respectively defined by 
\begin{eqnarray*}
J_{t}^{1}&=&
\Ga_{t}
\int_{0}^{t} 
\La_{\eta(s)}
\partial V V(y^{n}_{\eta(s)}) \delta x_{\eta(s), s} 
 \, d x_{s}
 \nonumber
 \\
 J_{t}^{2}&=& 
 \Ga_{t}
\int_{0}^{t} 
(\La_{s}-\La_{\eta(s)})
\partial V V(y^{n}_{\eta(s)}) \delta x_{\eta(s), s} 
 \, d x_{s}
 \nonumber
 \\
 J_{t}^{3}&=& 
\Ga_{t} \int_{0}^{t}\La_{s}  \, dI_{s}  \, ,
\end{eqnarray*}
and where $ J^{4}_{t}, J^{5}_{t}$ are given by
\begin{eqnarray*}
J_{t}^{4}&=& -
H\cdot \Ga_{t} \int_{0}^{t}(\La_{s}-\La_{\eta(s)}) \partial VV (y^{n}_{\eta(s)})  \cdot (s-\eta(s))^{2H-1}  \, ds 
\nonumber
\\
J_{t}^{5}&=& -
H\cdot \Ga_{t} \int_{0}^{t}\La_{\eta(s)} \partial VV (y^{n}_{\eta(s)})  \cdot (s-\eta(s))^{2H-1}  \, ds 
. 
\end{eqnarray*}
\end{lemma}

\begin{proof}  
We first recall that the continuous time Euler scheme defined in \eqref{eqn.euler3} can be written, for $s\in[0,T]$, as
\begin{eqnarray}\label{e.euler.1s}
\delta y^{n}_{\eta(s), s}  = V(y^{n}_{\eta(s)}) \delta x_{\eta(s), s} + \frac12 \partial VV (y^{n}_{\eta(s)}) \cdot (s-\eta(s))^{2H}, 
\end{eqnarray}
where we   $\eta_{s}=t_{k}$ for $s\in [t_{k}, t_{k+1})$.    
One can also write equation \eqref{e.euler.1s} in integral form, which yields an expression of the form: 
\begin{eqnarray}\label{eqn.euler4}
y^{n}_{t} = y_{0} + \int_{0}^{t} V(y^{n}_{\eta(s)}) dx_{s} - H \int_{0}^{t} \partial VV (y^{n}_{\eta(s)}) (s-\eta(s))^{2H-1} ds. 
\end{eqnarray} 
 Gathering \eqref{eqn.euler4} with equation \eqref{e2.1} for which we omit the drift term, we get 
\begin{multline}\label{e.error}
y_{t}-y^{n}_{t} = \int_{0}^{t} ( V(y_{s}) - V(y^{n}_{s}) ) \, d x_{s} +
\int_{0}^{t} ( V(y^{n}_{s}) - V(y^{n}_{\eta(s)}) ) \, d x_{s} 
\\  - H\cdot \int_{0}^{t} \partial VV(y^{n}_{\eta(s)}) \cdot (s-\eta(s))^{2H-1} \, ds 
.
\end{multline}
Next we will  consider a decomposition   of the quantity $V(y^{n}_{s}) - V(y^{n}_{\eta(s)}) $ in  \eqref{e.error}. 
Namely we   apply the chain rule twice to obtain   
\begin{equation}\label{e.V.1s}
V(y^{n}_{s}) - V(y^{n}_{\eta(s)}) 
= \int_{\eta(s)}^{s} \partial V (y^{n}_{u}) dy^{n}_{u}
= \partial V(y^{n}_{\eta(s)}) \delta y^{n}_{\eta(s), s} 
+
\int_{\eta(s)}^{s} \int_{\eta(s)}^{u}\partial^{2} V (y^{n}_{v}) dy^{n}_{v} dy^{n}_{u} . 
\end{equation}

Plugging \eref{e.euler.1s} into \eref{e.V.1s} and then integrating in $x$ we thus get
\begin{eqnarray}
\int_{0}^{t} ( V(y^{n}_{s}) - V(y^{n}_{\eta(s)}) ) d x_{s}  
&=& 
\int_{0}^{t} 
\partial V(y^{n}_{\eta(s)}) 
\lp
V(y^{n}_{\eta(s)}) \delta x_{\eta(s), s} + \frac12 \partial VV (y^{n}_{\eta(s)}) \cdot (s-\eta(s))^{2H}
\rp
 d x_{s} 
 \nonumber
\\
&&+
\int_{0}^{t} 
\lp
\int_{\eta(s)}^{s} \int_{\eta(s)}^{u}\partial^{2} V (y^{n}_{v}) dy^{n}_{v} dy^{n}_{u}
\rp
 d x_{s} ,\label{eqn.ymy}
\end{eqnarray}
Recalling the definition of $I_{t}^{1}$, $I_{t}^{2}$ in \eqref{eqn.I}, equation \eqref{eqn.ymy} 
 can also be read as 
\begin{equation} \label{eqn.dec1}
\int_{0}^{t} ( V(y^{n}_{s}) - V(y^{n}_{\eta(s)}) ) d x_{s}  
=
\int_{0}^{t} 
\partial V V(y^{n}_{\eta(s)}) \delta x_{\eta(s), s} 
 d x_{s}  +I_{t}^{1} +I_{t}^{2} . 
\end{equation}

We now decompose 
  the quantity  $V(y_{s}) - V(y^{n}_{s})$ in   \eqref{e.error}. Specifically   we write     
\begin{eqnarray} \label{eqn.dec2}
V(y_{t}) - V(y^{n}_{t}) = \int_{0}^{1} \partial V(\theta y_{t} +(1-\theta)y^{n}_{t}) d\theta \cdot (y_{t}-y^{n}_{t})
 =   \tilde{V} (t)  \cdot (y_{t}-y^{n}_{t}),    
\end{eqnarray}
where we recall that the process $\tilde{V}$ has been introduced in Definition \ref{def.v1}. 

We are ready to  plug \eqref{eqn.dec1} and \eqref{eqn.dec2} into \eqref{e.error}  in order to get the following linear equation for $y-y^{n}$: 
\begin{eqnarray}\label{eqn.ymyk}
 y_{t}-y^{n}_{t} = \int_{0}^{t} 
\tilde{V} (s)  \cdot (y_{s}-y^{n}_{s}) dx_{s} +K_{t}, 
\end{eqnarray}
where the process $K_{t}$ is given by 
\begin{eqnarray}\label{eqn.kd}
K_{t} = \int_{0}^{t} \partial VV(y^{n}_{\eta(s)}) \delta x_{\eta(s)s} dx_{s} + I^{1}_{t}+I^{2}_{t} -  H\cdot \int_{0}^{t}  \partial VV (y^{n}_{\eta(s)})  \cdot (s-\eta(s))^{2H-1} ds . 
\end{eqnarray}
Eventually we recall that $\Ga$ solves the Jacobian type equation \eqref{eqn.Ga} and that $ \La_{t} = \Ga_{t}^{-1}$. Hence applying   Duhamel's principle in order 
solve \eqref{eqn.ymyk}, we get 
\begin{eqnarray*}
y_{t}-y^{n}_{t}=  \Ga_{t} \int_{0}^{t}\La_{s} \, dK_{s}. 
\end{eqnarray*}
Thanks to our expression \eqref{eqn.kd}, the above equation can be written more explicitly as 
\begin{eqnarray}\label{eqn.ymy2}
y_{t}-y^{n}_{t}=  \Ga_{t} \int_{0}^{t}\La_{s} \partial VV(y^{n}_{\eta(s)}) \delta x_{\eta(s)s} dx_{s} 
+ 
 \Ga_{t} \int_{0}^{t}\La_{s} d ( I^{1}_{s}+I^{2}_{s})  
 \nonumber
 \\
 -  H\cdot  \Ga_{t} \int_{0}^{t}\La_{s}   \partial VV (y^{n}_{\eta(s)})  \cdot (s-\eta(s))^{2H-1} ds .   
\end{eqnarray}
With relation \eqref{eqn.ymy2} in hand, we can now easily identify the terms in \eqref{eqn.ymyj}. Indeed, the second term on    the right-hand side of equation \eqref{eqn.ymy2} is exactly $J^{3}_{t}$.  
 Also, in the same equation,  by plugging   the decomposition $\La_{s} = \La_{\eta(s)} + (\La_{s} - \La_{\eta(s)})$ into 
   the first and third  terms    we    identify the first  and third  terms as $J_{t}^{1}+J_{t}^{2} $ and $J_{t}^{4}+J_{t}^{5}$, respectively.  We thus conclude the  identity \eqref{eqn.ymyj}. The proof is complete. 
\end{proof}

\subsection{The weak convergence of the Euler scheme}\label{section.weak}

We can now gather all the previous preliminary estimates in order to obtain our main result. This is summarized in the theorem below. 
\begin{thm}
Consider a vector field $V\in C^{4}_{b}$ and a driving fBm $x$ with Hurst parameter $H>1/3$. Let $y$ be the solution of the rough differential equation \eqref{e2.1}. The corresponding interpolated Euler scheme is $y^{n}$, displayed in \eqref{eqn.euler3}.    
 Then 
for any $f\in C^{4}_{b}(\RR^{d})$ and $t\in [0,T]$ there is a constant $C>0$ independent of $n$ such that
\begin{eqnarray}\label{e.weak}
\left|\mE
f(y^{n}_{t}) -\mE f(y_{t})
\right|
\leq \frac{C}{n^{4H-1-\ep}}. 
\end{eqnarray}

\end{thm}
\begin{proof}
For   conciseness we will prove the theorem for the case $V_{0}\equiv 0$ only. The general case can be considered in the similar way and is left to the patient reader.    

Let $f$ be a generic $C_{b}^{4}$ function. For $t\in [0, T]$ we define an interpolated process 
\begin{eqnarray}\label{e.f1}
f_{1}(t) = \int_{0}^{1} \partial f (\la y_{t}+(1-\la)y^{n}_{t}) d\la. 
\end{eqnarray}
Here note that in order to alleviate notations, we still drop indices and perform our computation as if our quantities were real-valued. Next a simple application of the fundamental theorem of calculus plus Lemma   \ref{prop.decomposition} reveal that
\begin{eqnarray}\label{eqn.dfy}
f(y_{t})-f(y^{n}_{t})&=&f_{1}(t) \cdot (y_{t} - y^{n}_{t})  =  \sum_{e=1}^{5}  f_{1}(t) J_{t}^{e}. 
\end{eqnarray}
The remainder of the proof is dedicated to estimate the five terms   in the right-hand side of~\eqref{eqn.dfy}. For sake of conciseness we   prove \eqref{e.weak} for   $t\in \ll0,T\rr$ only.  The proof for   $t\in [0,T]$ follows the same lines and is left to patient reader. 


\noindent \emph{Step 1: Estimating $J_{t}^{1}$ and $J_{t}^{5}$.} 
In this step, we consider the first and fifth term in \eqref{eqn.dfy}. 
Note that the integrals in the expressions for   $J_{t}^{1}$ and $J_{t}^{5} $ are in fact discrete sums. We can thus combine those two terms in order to get:  
\begin{eqnarray}\label{eqn.j1j5}
f_{1}(t) (  J_{t}^{1}+J_{t}^{5} )  &=& \sum_{t_{k}<t} f_{1}(t)   \Ga_{t} \Lambda_{t_{k}} \partial VV (y^{n}_{t_{k}}) 
\left( x^{2}_{t_{k}t_{k+1} } - \frac12 \Delta^{2H}   \right)
 .
\end{eqnarray}

Let us say a few words about the term $\psi_{k} \equiv x^{2}_{t_{k} t_{k+1} }-\frac12 \Delta^{2H}$ in the right-hand side of \eqref{eqn.j1j5}. First we highlight again the fact that we are performing 1-d type computations in order to simplify notation. In a $d$-dim setting we would consider random variables of the form 
\begin{eqnarray*}
\psi^{ij}_{k} = x^{2,ij}_{t_{k}  t_{k+1} } - \frac12 \Delta^{2H} \mathbf{1}_{\{i=j\}},
\quad\text{for}\quad i,j\in \{1,\dots, d\}. 
\end{eqnarray*}
Here we will just focus on the terms $\psi_{k}\equiv\psi^{ii}_{k}$, which are the most demanding ones. We leave the off-diagonal terms $\psi^{ij}_{k}$ to the patient reader for sake of conciseness. Next we should also have in mind the fact that $\psi_{k}$ can be written as 
\begin{eqnarray*}
\psi_{k} = \frac12 (\delta x_{t_{k}  t_{k+1} })^{2} - \frac12 \Delta^{2H} = \frac12 \Delta^{2H} H_{2} \lp \frac{(\delta x_{t_{k}  t_{k+1} })^{2}}{\Delta^{2H}} \rp, 
\end{eqnarray*}
where $H_{2}$ stands for the Hermite polynomial $H_{2}(x) = x^{2}-1$. Invoking \cite[Page 23]{N}
we thus get 
\begin{eqnarray*}
\psi_{k} =  \delta^{\diamond, 2} (\be_{t_{k} t_{k+1} }), 
\end{eqnarray*}
where $\be_{t_{k}t_{k+1}}$ is defined by \eqref{eqn.beta} and $\delta^{\diamond, 2}$ stands for a double Skorohod integral (see Section~\ref{subsection.d}  for Malliavin calculus notation). Hence applying twice the integration by parts \eqref{eqn.it}, we end up with  
\begin{eqnarray*}
\mE \Big[
f_{1}(t) (  J_{t}^{1}+J_{t}^{5} )\Big]
  = \sum_{t_{k}<t} 
 \mE 
 \left[\left\langle
  D^{2}\big[ f_{1}(t)  
   \Ga_{t} \Lambda_{t_{k}} \partial VV (y^{n}_{t_{k}}) \big], \be_{t_{k}  t_{k+1} }
   \right\rangle_{\ch^{\otimes 2}}
   \right] \, ,
\end{eqnarray*}
where recall that $\be$ is defined in \eqref{eqn.beta}. 
Applying Lemma \ref{lemma.d2f} with $\varphi$  given by 
$$\vp=D^{2}\big[ f_{1}(t)  
   \Ga_{t} \Lambda_{t_{k}} \partial VV (y^{n}_{t_{k}}) \big] \, ,
   $$ 
   and recalling that $f_{1}$ is the process in~\eqref{e.f1},
we   obtain
\begin{eqnarray}\label{e.f1bd}
\Big|
\mE \Big[
f_{1}(t) (  J_{t}^{1}+J_{t}^{5} )\Big]
\Big|
\leq  \sum_{t_{k}<t} 
 n^{-4H}
\mE \lc \left\|
  D^{2}\big[ f_{1}(t)  
   \Ga_{t} \Lambda_{t_{k}} \partial VV (y^{n}_{t_{k}}) \big] 
   \right\|_{\infty}
\rc\,  .
\end{eqnarray}
The  integrability results Theorem \ref{thm.ixi} and Lemma \ref{lem.ig} (b)  guarantee  the uniform  integrability in $n$ of the sup-norm in the inequality \eqref{e.f1bd}. Therefore,   we have the estimate    
\begin{eqnarray}\label{e.f1j}
\Big|
\mE \Big[
f_{1}(t) (  J_{t}^{1}+J_{t}^{5} )\Big]
\Big|
\leq  C  \sum_{t_{k}<t} 
 n^{-4H} = C 
 n^{1-4H} . 
\end{eqnarray}

\noindent \emph{Step 2: Estimating $J_{t}^{2}$.} 
We turn to the estimate of $J_{t}^{2}$ in \eqref{eqn.dfy} and Lemma \ref{prop.decomposition}. 
 Observe that according to the 
 fact that $\La=\Ga^{-1}$ and   $\Lambda$ solves \eqref{eqn.Ga}, we have  $ \La_{s} - \La_{\eta(s)} = - \int_{\eta(s)}^{s} \Lambda_{u}\tilde{V}(u) dx_{u} $. Substituting this into $J_{t}^{2}$ we obtain
\begin{eqnarray}\label{e.j2}
J_{t}^{2}  &=& 
 -\Ga_{t}
\int_{0}^{t} 
\Big(
\int_{\eta(s)}^{s} \Lambda_{u}\tilde{V}(u) dx_{u}
\Big)
\partial V V(y^{n}_{\eta(s)}) \delta x_{\eta(s)  s} 
 d x_{s}
  .
\end{eqnarray}
Now let us write 
\begin{eqnarray*}
\Lambda_{u}\tilde{V}(u)  = \big(
\Lambda_{u}\tilde{V}(u) - 
\Lambda_{\eta(s)}\tilde{V}(\eta(s)) 
\big) + \Lambda_{\eta(s)}\tilde{V}(\eta(s)) . 
\end{eqnarray*}
Reporting this relation  into our expression \eqref{e.j2} for  $J_{t}^{2}$ yields the   decomposition:  
\begin{eqnarray}
J_{t}^{2}  &=& 
    -\Ga_{t}
\int_{0}^{t} 
\int_{\eta(s)}^{s} \big(
\Lambda_{u}\tilde{V}(u) - 
\Lambda_{\eta(s)}\tilde{V}(\eta(s)) 
\big)
dx_{u} \cdot
\partial V V(y^{n}_{\eta(s)}) \delta x_{\eta(s)  s} 
 d x_{s}
 \nonumber
 \\
&& -
 \sum_{t_{k}<t}
 \Ga_{t}
\int_{t_{k}}^{t_{k+1}} \Lambda_{t_{k}}\tilde{V}(t_{k}) 
 \delta x_{t_{k}s} 
\partial V V(y^{n}_{t_{k}})  \delta x_{t_{k}  s} 
 d x_{s}
=: 
J_{t}^{21}+J_{t}^{22} . 
\label{e.j2d}
\end{eqnarray}
We now proceed to the analysis of $J^{21}_{t}$ and $J^{22}_{t}$ above.

In order to bound the term $J_{t}^{22}$ in our decomposition \eqref{e.j2d}, observe  that this term is of the form $\sum_{t_{k}<t} f_{t_{k}} \delta g_{t_{k}t_{k+1}}$ as in Lemma \ref{lem.y}.   Precisely, we have
\begin{eqnarray}\label{e.j22}
J_{t}^{22} = -
 \sum_{t_{k}<t}
 \underbrace{
 \Ga_{t} \Lambda_{t_{k}}\tilde{V}(t_{k}) 
\partial V V(y^{n}_{t_{k}}) 
}_{= f_{t_{k}}}
\cdot
\underbrace{
\int_{t_{k}}^{t_{k+1}}  \delta x_{t_{k}s} 
 \delta x_{t_{k}  s} 
 d x_{s} 
 }_{= \delta g_{t_{k}t_{k+1}}}
 . 
\end{eqnarray}
Moreover, according to \eqref{eqn.gala} and the $L^{p}$-estimates for $\cm_{0}$, $\cm_{1}$, it is readily checked that for all $p\geq 1$ and $(u,v)\in \cs_{2}(\ll0,T\rr)$ we have 
\begin{eqnarray}\label{e.dfb}
\lp \mE \lc |\delta f_{uv}|^{2p} \rc \rp^{\frac{1}{2p}} \lesssim |v-u|^{H-\ep}.
\end{eqnarray}
In addition $g$ has to be seen as a triple iterated integral of $x$. 
It has been shown in \cite[Lemma 4.3]{LT} that for all $(u,v)\in \cs_{2}(\ll0,T\rr)$ we have  
\begin{eqnarray}\label{e.dgb}
\lp \mE \lc |\delta g_{uv} |^{2p} \rc\rp^{\frac{1}{2p}} \lesssim \frac{|v-u|^{1/2}}{n^{3H-1/2}}. 
\end{eqnarray}
Since we are considering points $u,v$ on the grid $\ll0,T\rr$, it is readily checked that $v-u\geq T/n$. Hence one can play with the exponents in \eqref{e.dgb} and write 
\begin{eqnarray}\label{e.gb}
\lp \mE \lc |\delta g_{uv} |^{2p} \rc\rp^{\frac{1}{2p}} \lesssim \frac{|v-u|^{1-H+2\ep}}{n^{4H-1-2\ep}}. 
\end{eqnarray}
It follows that gathering \eqref{e.dfb} and \eqref{e.gb} one can apply Lemma   \ref{lem.y} to \eqref{e.j22} and get   
\begin{eqnarray}\label{eqn.j22}
|\mE [ f_{1}(t) J_{t}^{22} ] |\leq   \frac{C}{n^{4H-1-\ep}}.
\end{eqnarray}
In order to    bound   $\mE [ f_{1}(t) J_{t}^{21} ]$, where $J_{t}^{21}$ is defined in \eqref{e.j2d}, we need to make a further decomposition.  Using the product rule 
plus equation \eqref{eqn.Ga} for $\La$, Definition \ref{def.v1} for $\tilde{V}$, as well as relation \eqref{e2.1} and \eqref{eqn.euler3} for $y$ and $y^{n}$, 
 we can  write
 \begin{equation}\label{e.dlv}
 \Lambda_{u}\tilde{V}(u) - 
\Lambda_{\eta(s)}\tilde{V}(\eta(s))  
=
\int_{\eta(s) } ^{u}  f_{2} (v) dx_{v}+\int_{\eta(s) } ^{u}  f_{3} (v) d(v-\eta(v))^{2H} \, ,  
 \end{equation}
 where we have set
 \begin{eqnarray}
 f_{2} (v)&=&
 -\La_{v} \tilde{V}(v) \tilde{V}(v) +
\La_{v}   \partial \tilde{V}(v) \\
 f_{3} (v)&=&\frac14 
\Lambda_{v} \int_{0}^{1} \partial\partial V(\theta y_{v}+(1-\theta)y^{n}_{v}) (1-\theta)\partial VV(y_{\eta(s)}) \, ,
 \end{eqnarray}
and where we denote 
\begin{eqnarray*}
 \partial \tilde{V}(v)   = \int_{0}^{1} \partial\partial V(\theta y_{v}+(1-\theta)y^{n}_{v}) (\theta V(y_{v})+(1-\theta) V(y^{n}_{\eta(v)})) d\theta  . 
\end{eqnarray*}
Then we write
\begin{eqnarray*}
\int_{\eta(s) } ^{u}  f_{2} (v) dx_{v}
= 
f_{2}(\eta(s)) \delta x_{\eta(s), u}
+
\int_{\eta(s) } ^{u} ( f_{2} (v) -f_{2}(\eta(s)) ) dx_{v}  . 
\end{eqnarray*}
Substituting the above into $J_{t}^{21}$ we obtain a weighted sum of 
   two 4th and one  5th order multiple integral in  
  the form   $ \sum_{0\leq t_{k}<t} h_{k} $. Precisely, we have   $J_{t}^{21} = -(J_{t}^{211} +J_{t}^{212}+J_{t}^{213})$, where 
\begin{eqnarray}\label{e.j21}
J^{211}_{t} &=&
  \sum_{0\leq t_{k}<t}\Ga_{t}
\int_{t_{k}}^{t_{k+1}\wedge t} 
\int_{t_{k}}^{s} 
f_{2}(t_{k}) \delta x_{t_{k}, u}
dx_{u} \cdot
\partial V V(y^{n}_{t_{k}}) \delta x_{t_{k} s} 
 d x_{s}
  \equiv
   \sum_{0\leq t_{k}<t}\Ga_{t} h^{211}_{t}
   \nonumber
 \\
J_{t}^{212} &=&   \sum_{0\leq t_{k}<t}\Ga_{t}
\int_{t_{k}}^{t_{k+1}\wedge t} 
\int_{t_{k}}^{s} 
\int_{t_{k} } ^{u} ( f_{2} (v) -f_{2}(t_{k}) ) dx_{v} 
dx_{u} \cdot
\partial V V(y^{n}_{t_{k}}) \delta x_{t_{k}  s} 
 d x_{s}
\nonumber
\\
  &\equiv&
 \sum_{0\leq t_{k}<t}\Ga_{t} h^{212}_{k} 
 \nonumber
 \\
 J^{213}_{t} &=& -\Ga_{t}
\int_{0}^{t} 
\int_{\eta(s)}^{s} 
\int_{\eta(s) } ^{u}  f_{3} (v) d(v-\eta(v))^{2H}
dx_{u} \cdot
\partial V V(y^{n}_{\eta(s)}) \delta x_{\eta(s)  s} 
 d x_{s}. 
\end{eqnarray}
We now proceed to estimate the terms $J^{211}_{t}$, $J^{212}_{t}$ and  $J^{213}_{t}$.

One can easily analyze the term $J^{211}_{t}$ by writing 
\begin{eqnarray*}
h^{211}_{k} = f_{2}(t_{k}) \partial VV (y^{n}_{t_{k}}) x^{4}_{t_{k}, t_{k+1}\wedge t} , 
\end{eqnarray*}
where $x^{4}_{st}$ denotes the  4th order iterated integral over the interval $[s,t]$. It follows that $\|h^{211}_{k}\|_{L^{p}}\lesssim \frac{1}{n^{4H}}$ for $p\geq 1$.  This implies that   
\begin{eqnarray}\label{e.j211bd}
\mE [ f_{1}(t)  J_{t}^{211} ]
\leq
\sum_{t_{k}<t} C \cdot n^{-4H} \leq \frac{C}{n^{4H-1}} . 
\end{eqnarray}
In the same way    we can show that  the   bound  \eqref{e.j211bd} also holds for $J^{213}_{t}$. 

As far as 
  $J_{t}^{212}$ is concerned, one can  recast the term $h^{212}_{k}$ as 
  $h^{212}_{k} = \partial VV(y^{n}_{t_{k}}) \hat{h}^{212}_{k}$, with   
  \begin{eqnarray}\label{e.hh212}
 \hat{h}^{212}_{k} = 
\int_{t_{k}}^{t_{k+1}\wedge t} 
\int_{t_{k}}^{s} 
\int_{t_{k} } ^{u} ( f_{2} (v) -f_{2}(t_{k}) ) dx_{v} 
dx_{u}  \delta x_{t_{k}  s} 
 d x_{s}.
\end{eqnarray}
The quantity $ \hat{h}^{212}_{k}$ has to be seen as a 5th order iterated integral. 
One way to quantify this is to resort to Fubini's theorem and write 
\begin{equation*}
\hat{h}^{212}_{k} = \int_{t_{k}}^{t_{k+1}} z^{t_{k}t_{k+1}}_{v} dx_{v}, 
\quad\text{with}\quad
z_{v}^{t_{k}t_{k+1}} = (f_{2}(v)-f_{2}(t_{k})) \int_{v}^{t_{k+1}} dx_{u} \int_{u}^{t_{k+1}} \delta x_{t_{k}s}dx_{s}.
\end{equation*}
Using the rough path property of $x$ recalled in Section \ref{section.rp} and the definition of $f_{2}$ in \eqref{e.dlv}, it is readily checked that $z^{t_{k}t_{k+1}}$ is of order $(1/n)^{4H-\ep}$ for any $\ep>0$. Reporting this information in \eqref{e.hh212}, one gets the almost sure relation 
\begin{eqnarray*}
|\hat{h}_{k}^{212}|\leq \frac{G}{  n^{5(H-\ep)}  }, 
\end{eqnarray*}
where $G\in \cap_{p\geq 1}L^{p}(\Omega)$. 
With relation \eqref{e.j21} in mind and taking into account the definition~\eqref{e.f1} of $f_{1}$, we discover that 
 \begin{eqnarray}\label{e.fj212}
 \Big|
\mE [ f_{1}(t)  J_{t}^{212} ]\Big|
\leq \frac{C}{n^{-1+5(H-\ep)}} 
  \leq \frac{C}{ n^{ 4H-1} } . 
\end{eqnarray}

Summarizing our considerations for the term $J^{2}_{t}$, we gather our estimates \eqref{e.j211bd} and \eqref{e.fj212}. This yields the desired estimate  
\begin{eqnarray}\label{e.f1j2}
\Big|\mE [ f_{1}(t) J_{t}^{2} ] \Big|\leq \frac{C}{  n^{ 4H-1-\ep} }. 
\end{eqnarray}

\noindent \emph{Step 3: Estimating $J_{t}^{3}$.} 
In this step, we consider 
the term 
$J_{t}^{3}$ defined in Lemma \ref{prop.decomposition}. Also recall that $I_{t}$ has been decomposed into $I^{1}_{t}+I^{2}_{t}$ in \eqref{eqn.I}. Accordingly we shall write
\begin{eqnarray*}
J^{3}_{t} = J^{31}_{t}+J^{32}_{t} \equiv \Ga_{t} \int_{0}^{t}\La_{s} dI^{1}_{s} + \Ga_{t} \int_{0}^{t}\La_{s} dI^{2}_{s} , 
\end{eqnarray*}
and estimate $J^{31}_{t}$, $J^{32}_{t}$ separately. 
Resorting to expression \eqref{eqn.I} for $I^{2}$, let us write $J^{32}_{t}$ as 
\begin{eqnarray*}
J^{32}_{t} = \Ga_{t} \int_{0}^{t} \int_{\eta(s)}^{s} \int_{\eta(s)}^{u} \La_{s} \partial^{2}V(y^{n}_{v}) dy^{n}_{v}dy^{n}_{u} dx_{s} . 
\end{eqnarray*}
In this way, it is readily checked that   $J^{32}_{t}$ exhibits the same type of regularity as $J^{2}_{t}$ defined by \eqref{e.j2}.  The complete analysis of $J^{32}_{t} $ thus follows the same steps as $J^{2}_{t}$. It relies on another discretization procedure, similar to \eqref{e.j2d}. Namely one writes $J^{32}_{t} = J^{321}_{t}+J^{322}_{t}$, with 
\begin{eqnarray*}
J^{321}_{t} &=& \Ga_{t} \int_{0}^{t} \int_{\eta(s)}^{s} \int_{\eta(s)}^{u} \lp
\La_{s} \partial^{2}V(y^{n}_{v}) - \La_{\eta(s)} \partial^{2}V(y_{\eta(v)}) 
\rp dy^{n}_{v}dy^{n}_{u}dx_{s}, 
\\
J^{322}_{t} &=& \Ga_{t}  \sum_{t_{k}<t} \La_{t_{k}}\partial^{2}V(y_{t_{k}})   
\int_{t_{k}}^{t_{k+1}}
 \int_{\eta(s)}^{s} \int_{\eta(s)}^{u}   dy^{n}_{v}dy^{n}_{u}dx_{s} . 
\end{eqnarray*}
In addition, along the same lines as for \eqref{e.j22} and resorting to the discrete dynamics \eqref{eqn.euler3} of $y^{n}$, one can express $J^{321}_{t}$ as a weighted sum of triple integrals of $x$. We can thus proceed as in the  estimation of $J^{2}_{t}$ and get the same inequalities as in \eqref{eqn.j22}, \eqref{e.j211bd} and \eqref{e.fj212}. Details are left to the reader for sake of conciseness. We obtain  
 \begin{eqnarray}\label{e.j2bd}
\Big|  \mE [ f_{1}(t) J^{32}_{t}  ] \Big| \leq \frac{C}{ n^{4H-1-\ep}}. 
\end{eqnarray}

In order to bound the term $J^{31}_{t}$, we first use another step of discretization. That is we decompose $J^{31}_{t}$ as $J^{311}_{t} +J^{312}_{t}$ with 
\begin{eqnarray*}
J_{t}^{311}
&=&
\frac12  
  \Ga_{t} \int_{0}^{t}\La_{\eta(s)} 
\partial V    \partial VV (y^{n}_{\eta(s)}) \cdot (s-\eta(s))^{2H}
 d x_{s}  \\
J_{t}^{312}
&=&
\frac12  
  \Ga_{t} \int_{0}^{t}
  (
  \La_{s} 
 -
  \La_{\eta(s)} 
  )
\partial V    \partial VV (y^{n}_{\eta(s)}) \cdot (s-\eta(s))^{2H}
 d x_{s}  \, .
\end{eqnarray*}
Note that by applying Lemma \ref{lem.y} we can bound  $\mE[ f_{1}(t) J_{t}^{311} ]$    by  $\frac{1}{n^{4H-1-\ep}}$. Indeed, we can write
\begin{eqnarray}\label{e.fj311}
f_{1}(t)J_{t}^{311}
&=&
\frac12  
f_{1}(t)  \Ga_{t} \sum_{0\leq t_{k}<t} 
\La_{t_{k}} 
\partial V    \partial VV (y^{n}_{t_{k}}) \cdot \nu_{t_{k}, t_{k+1}\wedge t} , 
\end{eqnarray}
where the increment $\nu$ is defined by 
\begin{eqnarray}\label{e.nu}
\nu_{uv} = \int_{u}^{v}(s-\eta(s))^{2H}
 d x_{s} . 
\end{eqnarray}
Next recall the following result from Lemma 4.6 in \cite{LT}: For a fBm $x$ with Hurst parameter $H$ and $f$ such that $\|f\|_{\ga}\in L^{p}$ for all $\ga<H$ and $p\geq 1$,  we have 
\begin{eqnarray}\label{e.fnu}
\lcl \mE \lc \Big| \sum_{0\leq t_{k}<t} f_{t_{k}} \nu_{t_{k}, t_{k+1}\wedge t} \Big|^{p} \rc \rcl^{1/p} \leq \frac{CT}{n^{4H-1-\ep}}. 
\end{eqnarray}
One can apply directly this estimate to \eqref{e.fj311}  in order  to get
\begin{eqnarray}\label{e.fj311bd}
\mE[f_{1}(t)J_{t}^{311}]\leq \frac{C}{n^{4H-1-\ep}} . 
\end{eqnarray}


The term $J^{312}_{t}$ has to be compared to $J^{21}_{t}$ in \eqref{e.j2d}. We can thus follow some computations which are very similar to \eqref{e.dlv}-\eqref{e.j21}. We end up with 2nd and 3rd integrals involving $x$ and the increment $\nu$ in \eqref{e.nu}. Having the regularity \eqref{e.fnu} of $\nu$ into account we let the reader check that 
\begin{eqnarray}\label{e.fj312}
\mE[ f_{1}(t) J_{t}^{312} ] 
\leq \frac{C}{ n^{4H-1-\ep}},  
\end{eqnarray} 
similarly to \eqref{e.j211bd} and \eqref{e.fj212}. 
We can thus conclude this step by gathering \eqref{e.fj311bd} and \eqref{e.fj312}. This yields    
 \begin{eqnarray}\label{e.fj3bd}
|  \mE [ f_{1}(t) J_{t}^{3}  ] | \leq C n^{1-4H+\ep}. 
\end{eqnarray}

\noindent \emph{Step 4: Estimating $J_{t}^{4}$.} 
We now turn to an estimate of the term $J^{4}_{t}$ in Lemma \ref{prop.decomposition}. According to the expression therein and equation \eqref{eqn.Ga} for $\La$, observe that 
\begin{eqnarray}\label{e.f1j4}
\mE [ f_{1}(t)   J_{t}^{4} ] &=& -
H\cdot  \int_{0}^{t}
 Q^{t}_{s}
 \cdot (s-\eta(s))^{2H-1} ds ,  
\end{eqnarray}
where the quantity $Q^{t}_{s}$ is given by 
\begin{eqnarray*}
Q_{s}^{t} = \mE \lc f_{1}(t)
\lp
\int_{\eta(s)}^{s} 
  \Ga_{t} 
\La_{u} \tilde{V}(u)^{T}
 dx_{u} 
 \rp
\partial VV (y^{n}_{\eta(s)}) 
\rc. 
\end{eqnarray*}
As for $J_{t}^{31}$ we can show that
\begin{eqnarray}\label{e.fi}
\left|
\mE \lc f_{1}(t)
 \Ga_{t}\lp
\int_{\eta(s)}^{s}  
\La_{u} \tilde{V}(u)^{T}
 dx_{u} 
 \rp
\partial VV (y^{n}_{\eta(s)}) 
\rc
\right| \leq \frac{C}{ n^{ 2H-\ep}}. 
\end{eqnarray}
Indeed, by writing  and substituting
\begin{equation*}
\La_{u} \tilde{V}(u)^{T} = \La_{\eta(s)} \tilde{V}(\eta(s))^{T}+ (\La_{u} \tilde{V}(u)^{T}- \La_{\eta(s)} \tilde{V}(\eta(s))^{T})
\end{equation*}
into \eqref{e.fi}, we decompose \eqref{e.fi} into two components. The second   component obtained is a double integral over the interval $[\eta(s), s]$, which is bounded by $\frac{1}{n^{2H-\ep}}$. On the other hand, the first component is of the form $\mE[F \delta x_{\eta(s)s} ]$, where 
\begin{eqnarray*}
F= 
f_{1}(t)
 \Ga_{t}\lp
\La_{\eta(s)} \tilde{V}(\eta(s))^{T}
 \rp
\partial VV (y^{n}_{\eta(s)}) .
\end{eqnarray*}
Note that 
$F$ is an integrable variable whose Malliavin derivative $DF$ is also integrable. So applying integration by parts to $\mE[F \delta x_{\eta(s)s} ]$ and then Lemma \ref{lemma.bdd1} with $\varphi = DF$, together with the upper-bound estimates in Lemma \ref{lem.ig} and  Theorem \ref{thm.xirr},   we obtain the bound  $\frac{1}{n^{2H }}$.
Gathering those consideration and \eqref{e.fi} into \eqref{e.f1j4}, we end up with  
\begin{eqnarray}\label{e.f1j4bdd}
\mE [ f_{1}(t)   J_{t}^{4} ]  \leq \frac{C}{ n^{ 4H-1-\ep} }. 
\end{eqnarray}

\noindent {\it Step 5: Conclusion.}  Taking expectations on both sides of \eqref{eqn.dfy} and reporting \eqref{e.f1j}, \eqref{e.f1j2}, \eqref{e.fj3bd}, and \eqref{e.f1j4bdd} we discover that \eqref{e.weak} holds true. This finishes the proof.  \end{proof}

\section*{Acknowledgements}
\noindent
Y. Liu is partially supported by PSC-CUNY Award \# 66385-00 54.
S. Tindel is partially supported by NSF grants  DMS-1952966 and DMS-2153915.


\end{document}